\documentclass[11pt]{article}
\usepackage{fullpage}
\usepackage[english]{babel}
\usepackage[utf8x]{inputenc}
\usepackage[T1]{fontenc}
\usepackage{amsmath}
\usepackage{amssymb}
\usepackage{amsthm}
\usepackage{bbm}
\usepackage{parskip}
\usepackage{thm-restate}
\usepackage{booktabs}
\usepackage{array}

\usepackage[lined,boxed,ruled,norelsize,algo2e,linesnumbered]{algorithm2e}

\usepackage{hyperref}
\hypersetup{
	colorlinks=true,
	linkcolor=blue!70!black,
	citecolor=blue!70!black,
	urlcolor=blue!70!black
}
\usepackage{cleveref}
\usepackage{graphicx}

\newtheorem{theorem}{Theorem}
\newtheorem{lemma}{Lemma}
\newtheorem{fact}{Fact}
\newtheorem{claim}{Claim}

\newtheorem{definition}{Definition}
\newtheorem{corollary}{Corollary}
\newtheorem{proposition}{Proposition}
\newtheorem{example}{Example}

\newtheorem{problem}{Problem}
\newtheorem{assumption}{Assumption}
\allowdisplaybreaks

\newcommand{\defeq}{:=}

\newcommand{\norm}[1]{\left\lVert#1\right\rVert}
\newcommand{\norms}[1]{\lVert#1\rVert}

\newcommand{\inprod}[2]{\left\langle#1, #2\right\rangle}

\newcommand{\eps}{\epsilon}
\newcommand{\lam}{\lambda}

\newcommand{\xset}{\mathcal{X}}

\newcommand{\R}{\mathbb{R}}
\newcommand{\C}{\mathbb{C}}
\newcommand{\N}{\mathbb{N}}
\newcommand{\diag}[1]{\textbf{\textup{diag}}{\left(#1\right)}}
\newcommand{\half}{\frac{1}{2}}

\newcommand{\E}{\mathbb{E}}

\newcommand{\Var}{\textup{Var}}

\newcommand{\Nor}{\mathcal{N}}

\newcommand{\id}{\mathbf{I}}

\newcommand{\dd}{\textup{d}}
\usepackage{xcolor}
\definecolor{burntorange}{rgb}{0.8, 0.33, 0.0}

\newcommand{\poly}{\textup{poly}}
\newcommand{\polylog}{\textup{polylog}}
\newcommand{\Par}[1]{\left(#1\right)}
\newcommand{\Brack}[1]{\left[#1\right]}
\newcommand{\Brace}[1]{\left\{#1\right\}}

\newcommand{\Cov}{\mathbf{Cov}}

\newcommand{\ma}{\mathbf{A}}
\newcommand{\mb}{\mathbf{B}}
\newcommand{\mc}{\mathbf{C}}

\newcommand{\mpp}{\mathbf{P}}

\newcommand{\mr}{\mathbf{R}}

\newcommand{\va}{\mathbf{a}}

\newcommand{\vc}{\mathbf{c}}

\newcommand{\vh}{\mathbf{h}}

\newcommand{\vm}{\mathbf{m}}

\newcommand{\vr}{\mathbf{r}}
\newcommand{\vs}{\mathbf{s}}

\newcommand{\vv}{\mathbf{v}}
\newcommand{\vw}{\mathbf{w}}
\newcommand{\vx}{\mathbf{x}}
\newcommand{\vy}{\mathbf{y}}
\newcommand{\vz}{\mathbf{z}}

\newcommand{\msig}{\boldsymbol{\Sigma}}

\newcommand{\calA}{\mathcal{A}}

\newcommand{\calD}{\mathcal{D}}

\newcommand{\calI}{\mathcal{I}}

\newcommand{\calM}{\mathcal{M}}
\newcommand{\calN}{\mathcal{N}}

\newcommand{\calP}{\mathcal{P}}

\newcommand{\calS}{\mathcal{S}}

\newcommand{\calX}{\mathcal{X}}
\newcommand{\calY}{\mathcal{Y}}

\newcommand{\sfK}{\mathsf{K}}

\newcommand{\sfP}{\mathsf{P}}

\newcommand{\tilt}{\mathcal{T}}
\newcommand{\vhi}{\varphi}
\newcommand{\vvhi}{\boldsymbol{\vhi}}
\newcommand{\dkl}{D_{\textup{KL}}}
\newcommand{\dtv}{D_{\textup{TV}}}

\newcommand{\tpi}{\tilde{\pi}}

\newcommand{\gap}{\mathsf{gap}}
\newcommand{\Ferm}{F_{\textup{erm}}}
\newcommand{\Fpop}{F_{\textup{sco}}}
\newcommand{\simu}{\sim_{\textup{unif.}}}
\newcommand{\mecherm}{\calM_{\textup{erm}}}
\newcommand{\mechsco}{\calM_{\textup{sco}}}
\renewcommand{\epsilon}{\varepsilon}

\usepackage{mathtools}
\title{Functional Stochastic Localization}
\date{}
\author{Anming Gu\thanks{University of Texas at Austin, \texttt{anminggu@cs.utexas.edu}}
\and
Bobby Shi\thanks{University of Texas at Austin, \texttt{bhshi@utexas.edu}}
\and 
Kevin Tian\thanks{University of Texas at Austin, \texttt{kjtian@cs.utexas.edu}}}
\begin{document}

\maketitle

\begin{abstract}
Eldan's stochastic localization is a probabilistic construction that has proved instrumental to modern breakthroughs in high-dimensional geometry and the design of sampling algorithms. Motivated by sampling under non-Euclidean geometries and the mirror descent algorithm in optimization, we develop a functional generalization of Eldan's process that replaces Gaussian regularization with regularization by any positive integer multiple of a log-Laplace transform. We further give a mixing time bound on the Markov chain induced by our localization process, which holds if our target distribution satisfies a functional Poincar\'e inequality. Finally, we apply our framework to differentially private convex optimization in $\ell_p$ norms for $p \in [1, 2)$, where we improve state-of-the-art query complexities in a zeroth-order model.
\end{abstract}

\thispagestyle{empty}
\newpage
\tableofcontents
\thispagestyle{empty}
\newpage

\section{Introduction}\label{sec:introduction}
\setcounter{page}{1}
Stochastic localization, an elegant construction introduced in \cite{Eldan13}, has emerged as a key ingredient in modern analysis frameworks, enabling numerous recent breakthroughs in high-dimensional geometry and probability \cite{LeeV18,Eldan18,Klartag18,Eldan20,EldanMZ20,EldanS22,Guan24}. As a few highlights, the stochastic localization framework has led to improvements to the Kannan-Lov\'asz-Simonovitz (KLS) constant \cite{Eldan13,LeeV24,Chen21,klartagL22,JambulapatiLV22,Klartag23} and the recent resolution of Bourgain's slicing conjecture \cite{KlartagL25a} and the thin-shell conjecture \cite{KlartagL25}.
Recently, through its application to proximal sampling \cite{LeeST21, ChenCSW22} and the framework of localization schemes \cite{ChenE22}, stochastic localization has also seen applications in the design and analysis of sampling algorithms, for both continuous and discrete distributions \cite{EldanKZ22, ElAlaouiMS22, GopLL22, AnariHLVXY23}.

A standard viewpoint of stochastic localization is as solving the following problem: given a base density $\pi$ over $\R^d$, and a time parameter $\tau > 0$, construct a distribution over linear tilts $\vy_\tau \propto \R^d$,\footnote{Here and throughout, we use $\tilt_{\vv} \nu$ to mean the density that is $\propto \exp(\inprod{\vv}{\cdot})\nu$; see Section~\ref{sec:preliminaries} for notation.} such that $\pi_\tau \propto \tilt_{\vy_\tau}\exp(-\frac \tau 2 \norms{\cdot}_2^2)\pi$ equals $\pi$ in expectation, over the randomness of $\vy_\tau$. This equality holds in a pointwise sense, i.e., $\E[\pi_\tau(\vx)] = \pi(\vx)$ for all $\vx \in \R^d$. The reason this is beneficial is because regardless of how $\vy_\tau$ is realized, a large Gaussian regularization $\propto \exp(-\frac \tau 2 \norms{\cdot}_2^2)$ is induced. By using the strong geometric properties enjoyed by Gaussians, e.g., isoperimetric, Poincar\'e, or log-Sobolev inequalities, this process can then be used to reason about $\pi$ itself.

\paragraph{Functional stochastic localization.} A key feature of stochastic localization is that it primarily lends itself to tools tailored to the Euclidean geometry. This is unsurprising: Gaussians are induced by a (Euclidean) quadratic potential $\psi = \half\norm{\cdot}_2^2$, and various alternate definitions of stochastic localization (cf.\ \cite{ShiTZ25} for a recent survey) rely on stochastic differential equations (SDEs), which are based on the Euclidean construction of Brownian motion. This begs the question: is there a generalization that is more naturally equipped to handle non-Euclidean geometry, e.g., induced by a non-Euclidean norm, or even an arbitrary convex function $\psi: \R^d \to \R$?

This line of inquiry is motivated by the rich literature on non-Euclidean optimization, which has had wide-ranging applications such as sparse recovery \cite{CandesRT06}, matrix completion \cite{AgarwalNW10}, combinatorial optimization \cite{KelnerLOS14,KyngPSW19}, multi-armed bandit problems \cite{BubeckC12}, and more. Many of these applications are intimately related to the framework of \emph{mirror descent}, a ``functional'' generalization of gradient descent induced by a non-Euclidean regularizer $\psi: \R^d \to \R$. A parallel question would then be: is there a functional generalization of stochastic localization, e.g., a distribution over linear tilts $\vy_\tau \in \R^d$ that satisfies, for some convex $\psi: \R^d \to \R$, and $\pi$ a density on $\R^d$,
\begin{equation}\label{eq:fsl}\E\Brack{\pi_\tau} = \pi, \text{ pointwise on } \R^d,\text{ where } \pi_\tau \defeq \tilt_{\vy_\tau}\exp{\Par{-\tau\psi}}\pi?\end{equation}

The primary contribution of our work is to investigate this question, and to construct a generalized process realizing \eqref{eq:fsl} for structured $\psi$. We term our construction \emph{functional stochastic localization} (to contrast with \cite{Eldan13}, henceforth referred to as \emph{(Euclidean) stochastic localization}). 

\paragraph{Applications to sampling.} While we believe our construction is interesting in its own right, our primary motivation is applications to designing sampling algorithms under non-Euclidean regularity.
Since the characterization of sampling as optimization over the space of measures \cite{JordanKO98}, sampling algorithms for continuous distributions have drawn heavy inspiration from convex optimization \cite{Dalalyan16,Dalalyan17,Wibisono18,Wibisono19,VempalaW19,DurmusMM19,MaCCFBJ21,NitandaWS22} (cf.\ \cite{Chewi23} for a thorough discussion).
However, compared to sampling in Euclidean spaces, development of non-Euclidean samplers has lagged behind. The sampling analog of mirror descent is to reweight Brownian motion by $\nabla^2 \psi$, e.g., in the mirror Langevin dynamics \cite{HsiehKRC18,ChewiLLMRS20}. 
Several works have considered the corresponding discretization \cite{HsiehKRC18,ZhangPFP20,AhnC21,Jiang21,LiTVW22}; however, they often require non-standard assumptions 
to achieve vanishing bias \cite{LiTVW22}, or only discretize the drift but not the diffusion \cite{AhnC21,Jiang21,GuK25}. Other approaches include the discretized Riemannian Langevin dynamics, which requires exact sampling from the manifold \cite{GatmiryV22,LiE23}, or using a filter to obtain an unbiased algorithm \cite{SrinivasanWW24,SrinivasanWW25}, which again requires additional regularity.

Recently, \cite{GopiLLST25} introduced a complementary approach via a functional variant of the proximal sampler (an algorithmic variant of Euclidean stochastic localization), and studied its use in sampling from non-Euclidean spaces. This approach was based on the theory of the \emph{log-Laplace transform} $\vhi^\sharp$ (LLT, Definition~\ref{def:llt}), a probabilistic analog of the \emph{Fenchel transform} $\vhi^*$ central to mirror descent. Moreover, \cite{GopiLLST25} showed that LLTs satisfy useful properties, e.g., self-concordance and strong convexity-smoothness duality (Facts~\ref{fact:llt-derivatives},~\ref{fact:llt-smooth-convex}) that make them amenable to applications similar to uses of mirror descent. 
The construction of \cite{GopiLLST25} can be rephrased in terms of \eqref{eq:fsl}, where $\psi$ is an LLT, but their framework only permitted taking $\tau = 1$. 
Moreover, the proof strategy of \cite{GopiLLST25} led to a quadratic overhead in the mixing time of their sampler compared to Euclidean counterparts, opening the door to a stronger and more flexible approach.\footnote{We expand on this point in Section~\ref{ssec:approach}. The conference version appeared at COLT 2023 \cite{GopiLLST23}, but contained an error. This was later fixed in the third arXiv version, which led to a quadratic overhead, cf.\ their Section 1.4.}

\subsection{Our approach}\label{ssec:approach}

We develop a new approach to the construction of processes matching \eqref{eq:fsl}, when $\psi \defeq \vhi^\sharp$ is a log-Laplace transform (Definition~\ref{def:llt}). Our construction is summarized and analyzed in Section~\ref{sec:functional_stochastic_localization}, and carefully uses structural facts about LLTs and \emph{exponential families}, a parametric family of high-dimensional distributions, to argue that the martingale property \eqref{eq:fsl} holds.

\paragraph{Defining functional stochastic localization.} Our framework is motivated by and related to \cite{GopiLLST25}, but substantially broadens its flexibility. To better explain this relationship, it is helpful to outline a typical mixing-time result for Euclidean proximal samplers. Let $\pi$ be a density on $\R^d$ that satisfies an $\alpha$-Poincar\'e inequality (PI, Definition~\ref{def:vpi}). The proximal sampler of \cite{LeeST21} can be viewed as performing Gibbs sampling on an extended space, defined by a random tilt $\vy_\tau$ at time $\tau > 0$, and a particle $\vx \in \R^d$. The mixing time of the proximal sampler (in $\chi^2$ divergence) then scales $\propto \frac \tau \alpha$ \cite{LeeST21, ChenCSW22, ChenE22}. This leads to a tradeoff: larger $\tau$ leads to slower mixing, but induces additional, helpful, Gaussian regularization in the tilted densities $\pi_\tau$. Choosing $\tau$ to balance the cost of Gibbs sampling with the mixing time then enables a host of applications.

The ability to handle a range of $\tau$ is thus clearly useful. Unfortunately, previous approaches to stochastic localization were firmly continuous, letting $\tau \in \R$ index time in a SDE, and defining $\vy_\tau$ via infinitesimal convolution. The challenge in extending this approach beyond the Euclidean setting is that under minimal structure, infinitesimal convolutions necessarily lead to Gaussians by the central limit theorem. Previously, \cite{GopiLLST25} sidestepped this issue by only permitting $\tau = 1$; moreover, their main result (cf.\ Theorem 1, \cite{GopiLLST25}) has a mixing time that, in our language, scales with a suboptimal dependence $\propto \frac 1 {\alpha^2}$. 
Thus, it seems that a new perspective is necessary to achieve a nontrivial range of $\tau$ in non-Euclidean proximal sampling.

Our framework allows for taking arbitrarily large $\tau$, with a twist: although the underlying space $\R^d$ is continuous, $\tau \in \N$ in our process now indexes a discrete time. We then construct the process defining $\vy_\tau$ in \eqref{eq:fsl} using $\tau$ convolved draws from an exponential family (Algorithm~\ref{alg:localization}, see also the alternate definition of our process in \eqref{eq:llt-loc-process}). This construction is motivated by an information-theoretic viewpoint of stochastic localization as posterior sampling, given progressively denoised observations from a Gaussian channel \cite{AlaouiM22} (see also Section 3, \cite{ShiTZ25}). Our key observation is that, unlike the SDE-based definitions of stochastic localization, the perspective based on posterior sampling extends cleanly to discrete time, where $\tau \in \N$ is the number of observations. We formalize these ideas in Section~\ref{sec:functional_stochastic_localization} by leveraging properties of exponential families, where we prove that Algorithm~\ref{alg:localization} meets the requirement \eqref{eq:fsl} and is a localization process in the sense of \cite{ChenE22}.

\begin{algorithm2e}[t!]
\caption{Sampling from $\pi_\tau$}
\label{alg:localization}
\DontPrintSemicolon
\textbf{Input: }density $\pi$, $\psi: \R^d \to \R$ the LLT (Definition~\ref{def:llt}) of convex $\vhi: \R^d \to \R$ \;
$\vy_0\leftarrow \mathbf{0}$ \;
\For{$j=0, \dots, \tau-1$}{
    $\vz_j \sim \mathcal{T}_{\vy_j}\exp(-j\psi)\pi$ \; 
    $\vw_j\sim \mathcal{T}_{\vz_j}\exp(-\vhi)$\tcp*{When $\vhi = \psi = \half\norm{\cdot}_2^2$, this is equivalently $\vw_j \sim \Nor(\vz_j, \id_d)$}
    $\vy_{j+1}\leftarrow \vy_j+\vw_j$
}
$\vx \sim \mathcal{T}_{\vy_{\tau}}\exp(-\tau\psi)\pi$ \;
\textbf{Return:} $\vx$
\end{algorithm2e}

\paragraph{Non-Euclidean sampling.} The remainder of our paper analyzes the mixing time of a non-Euclidean proximal sampler (Algorithm~\ref{alg:alternate}). Our specific sampler arises by combining our functional stochastic localization process with the formalism of \cite{ChenE22}, and admits the following bound.

\begin{restatable}{theorem}{restatemixing}\label{thm:mixing}
Let $\vhi: \R^d \to \R$ be convex, $\psi \defeq \vhi^\sharp$, and $\pi$ satisfy an $\alpha$-$\psi$-Poincar\'e inequality (Definition~\ref{def:vpi}).  For any $\mu_0\ll \pi$, iterate $\vx_k$ of Algorithm \ref{alg:alternate} has density $\mu_k$ satisfying
    \begin{equation}\label{eq:convergence_result}
        \chi^2(\mu_k \|\pi) \le \frac{1}{(1 + \alpha/\tau)^{2k}}\chi^2(\mu_0\|\pi).
    \end{equation}
\end{restatable}

By taking $\vhi = \psi = \half\norm{\cdot}_2^2$ in Theorem~\ref{thm:mixing}, we recover standard convergence rates $\propto \frac \tau \alpha$ for the Euclidean proximal sampler (e.g., Theorem 4, \cite{ChenCSW22}). Interestingly, we  even obtain the sharpest constant known in the Euclidean case, by revisiting the analysis framework of \cite{ChenE22} in the context of Gibbs sampling and giving a sharper characterization of its convergence rate in Section~\ref{sec:gibbs_sampling}. Moreover, Theorem~\ref{thm:mixing} appealingly uses the minimal assumption of a $\psi$-PI, as opposed to the stronger condition of $\psi$-strong log-concavity (cf.\ Fact~\ref{fact:brascamp_lieb}). The property of $\psi$-PI is closed under Lipschitz transformations, tensorization, and  multiplicative perturbations (by the Holley-Stroock principle), so Theorem~\ref{thm:mixing} applies to non-log-concave densities that are appropriately compatible with $\psi$. An analogous statement relying on strong log-concavity would not admit such extensions; see \cite{VempalaW19,Chewi23} for more discussion of this point.  

In Appendix~\ref{app:proof_via_conductance}, we further give a mixing time bound that leverages our new process, but uses a conductance-based convergence analysis, as in \cite{GopiLLST25}. This strategy yields weaker estimates scaling as $\approx \max\{\frac \tau \alpha, \frac \tau {\alpha^2}\}$. Thus, achieving Theorem~\ref{thm:mixing} for the full range of $\tau \in \N$ and $\alpha > 0$ required both changing the process from \cite{GopiLLST25}, and the corresponding analysis.

In Section \ref{ssec:dp}, we showcase the applicability of our new localization framework by deriving new rates for the problem of differentially private (DP) convex optimization in $\ell_p$ geometries, for $p \in [1, 2)$. Our oracle query complexity bounds for this problem (Corollary~\ref{cor:dp_sample}) improve quadratically upon the prior work (Theorem 2, \cite{GopiLLST25}), due to Theorem~\ref{thm:mixing}. Moreover, if granted an appropriate warm start (stated formally in Assumption~\ref{assume:warmstart}), our resulting query complexity is optimal up to logarithmic factors for the entire range of $p$, as discussed in Appendix A, \cite{GopiLLST25}.

Our work takes an important step towards extending the theory of stochastic localization to more general geometries. Perhaps the main outstanding question we leave open is whether an entropic strengthening of Theorem~\ref{thm:mixing} (convergence in KL divergence, rather than $\chi^2$) holds, as it does in the Gaussian case (Theorem 3, \cite{ChenCSW22}). This could significantly improve mixing time estimates: for example, it would make our nearly-optimal query complexity in our main DP application unconditional. Unfortunately, entropic variants of key facts used in our proof, e.g., the Brascamp-Lieb inequality (Fact~\ref{fact:brascamp_lieb}), are known to be false in general, as \cite{BobkovL00} demonstrates.

Towards realizing these goals, we prove Theorem~\ref{thm:under-conjecture}, a conditional entropic mixing result, in Appendix~\ref{app:entropy}. Our proof may be of independent interest, as it is quite different from that of Theorem~\ref{thm:mixing}. This result holds under a condition (Assumption~\ref{assume:llt}) that is true for quadratic $\psi$, and could apply to a broader family. More generally, we outline directions for future exploration in Section~\ref{sec:conclusion}, including obtaining faster runtimes when our framework is induced by certain structured $\psi$.

\subsection{Prior work}
\paragraph{Proximal sampling.} Initial work on proximal methods for sampling \cite{Pereyra16,BrosseDMP17,Bernton18,Wibisono19} required stringent assumptions or had suboptimal mixing rates. More recently, \cite{LeeST21} proposed a ``modern'' proximal sampler for strongly log-concave targets that improves upon these issues. Subsequently, several works provide convergence guarantees in various divergences and under isoperimetric conditions \cite{ChenCSW22,AltschulerC24,MitraW25,LiangMW25,Wibisono25}. The proximal sampling framework has provided state-of-the-art results for structured sampling \cite{LeeST21,GopLL22,LiangC22,FanYC23,YuanFLWC23}, sampling from non-log-concave distributions \cite{HeMBE24,HeLSZ25}, and sampling from convex bodies \cite{KookVZ24,KookV25a,KookV25b,Kook25}. Notably, all of these results built upon the Euclidean proximal sampler; our work may permit similar developments in non-Euclidean settings.
We also mention the related work \cite{HeHBE24}, who also gave a variant of the proximal sampler that complements our motivation. Their focus was sampling from heavy-tailed densities, and their framework appears to not yield a localization process in the sense of \cite{ChenE22}. Moreover, their regularity assumptions still fundamentally cater to the Euclidean geometry.

\paragraph{Localization schemes.} As discussed previously, the ideas underpinning localization schemes were first introduced as Eldan's stochastic localization \cite{Eldan13}. Recently, Chen and Eldan \cite{ChenE22} unified stochastic localization and spectral independence, a method for analyzing discrete Markov chains, under the localization scheme framework, and it has been used to achieve improve guarantees for sampling from Markov chains \cite{ElAlaouiMS22,BentonDDD24,HuangMP24,AnariKV24,HuangMRW25,ChenLYZ25}. Other works that use the LLT for sampling via stochastic localization include \cite{EldanS22,AnariHLVXY23}.

Amongst extensions of the stochastic localization process in the literature, the most relevant to our construction is by \cite{MikulincerP24}, who also consider non-Gaussian reweightings in (approximate) stochastic localization. Their work leverages different structure than our main result: instead of regularizing the density by an LLT, they regularize via a low-degree polynomial. 

\paragraph{Renormalization.} Renormalization is yet another perspective on the derivation of functional inequalities, developed in the physics literature \cite{wilson1974renormalization, polchinski1984renormalization}; for a recent survey, see \cite{BauerschmidtBD24}.  The formalism associates with a measure a flow of measures in terms of a renormalized potential; this structure can be used to prove functional inequalities under a multiscale generalization of the Bakry-\'Emery criterion.  In fact, this perspective is dual to stochastic localization: at a high level, stochastic localization focuses on the stochastic evolution of a random measure, while renormalization focuses on a Markov process and its associated renormalized potential.  This framework has been successively employed to derive log-Sobolev inequalities in discrete systems, including in Euclidean field theory \cite{Kupiainen2016renormalization, serres2022behaviorpoincareconstantpolchinski} and in Ising models \cite{BAUERSCHMIDT2019simple, Bauerschmidt2024logsobolev}. We draw upon this literature in Section~\ref{sec:functional_stochastic_localization}.

\section{Preliminaries}\label{sec:preliminaries}

We use lowercase bold font, e.g., $\vx$ to denote vectors, and uppercase bold font, e.g., $\ma$ to denote matrices or linear operators. Uppercase letters, e.g., $X,Y$ denote random variables.  Throughout, $\tau \in \N$ will be reserved to denote a discrete time index and $\alpha>0$ will be reserved for strong convexity and Poincar\'e parameters.  Sans serif uppercase letters, e.g., $\mathsf{P}$ are used to denote Markov operators on suitable $L^2$ spaces, while sans serif lowercase letters, e.g., $\mathsf{p}$ are used to denote the corresponding probability kernels.

\paragraph{Probability.}
For a density $\pi$ supported on $\calX$, we define $\pi(S)\defeq \Pr_{\vx\sim \pi}[\vx\in S]$. Let $\mu, \nu$ be two densities such that $\mu \ll \nu$.
    The $\chi^2$-divergence of $\mu$ with respect to $\nu$ is $\chi^2(\mu \| \nu)\defeq \int (\frac{\dd \mu}{\dd\nu}-1)^2 \,\dd\nu$, and the Kullback-Leibler (KL) divergence of $\mu$ with respect to $\nu$ is $\dkl(\mu \| \nu)\defeq \int \log (\frac{\dd\mu}{\dd\nu}) \,\dd\mu$.

We use throughout the law of total variance: if $X, Y$ are random variables on the same probability space and the variance of $Y$ is finite, then
\begin{equation}\label{eq:total-var}
    \operatorname{Var}[Y]=\mathbb{E}[\operatorname{Var}[Y\mid X]]+\operatorname{Var}[\mathbb{E}[Y\mid X]].
\end{equation}

We use $\tilde{\pi}$ to denote general joint densities of random variables $(X, Y)$.  In this way, $\tilde{\pi}^X, \tpi^Y$ are the marginal densities, and $\tpi^{Y\mid X}, \tpi^{X\mid Y}$ are the conditional densities.

\paragraph{Optimization.} When $\norm{\cdot}$ is a norm on $\R^d$, $\norm{\cdot}_* \defeq \sup_{\norm{\vx}\le 1} \inprod{\vx}{\cdot}$ denotes the dual norm. We say that twice-differentiable $f: \R^d \to \R$ is $\alpha$-strongly convex with respect to $\norm{\cdot}$ if $\nabla^2 f(\vx)[\vv, \vv] \ge \alpha \norm{\vv}^2$ for all $\vx, \vv \in \R^d$, and if $\nabla^2 f(\vx)[\vv, \vv] \le \beta \norm{\vv}^2$, we say $f$ is $\beta$-smooth with respect to $\norm{\cdot}$. When $\ma$ and $\mb$ are matrices in finite dimension, we use $\ma \succeq \mb$ to mean that $\ma - \mb$ is positive semidefinite; when $f$ and $g$ are functions on $\R^d$, we use $f \succeq g$ to mean $f - g$ is convex.

\paragraph{Log-Laplace transforms.} A central object in our development is the log-Laplace transform.

\begin{definition}\label{def:llt}
    The \emph{log-Laplace transform (LLT)} of $\varphi:\R^d\to \R$ is defined as \[
    \varphi^\sharp(\vx) \defeq \log\Par{\int \exp\Par{\inprod{\vx}{\vy}-\varphi(\vy)}\dd \vy}.
    \]
We will typically use $\psi \defeq \vhi^\sharp$ as shorthand when $\varphi$ is fixed.
\end{definition}

For a measure $\nu$, define the tilted probability density at $\va$ as
\begin{equation}
        \mathcal{T}_\va \nu(\vx)\defeq \frac{\exp(\inprod{\va}{\vx}) \nu(\vx)}{\int \exp(\inprod{\va}{\vz}) \nu(\vz)\,\dd\vz}.
\end{equation}
In this way, \begin{equation}
    \tilt_\vx\exp(-\vhi(\vy)) = \exp\Par{\inprod{\vx}{\vy}-\varphi(\vy)-\psi(\vx)},
\end{equation}
which is indeed a probability density as the normalization constant is exactly given by $\exp(-\psi)$. More generally, if $f$ is any integrable nonnegative function over $\R^d$, we let $\tilt_{\vx} f$ be the density that is $\propto \exp(\inprod{\vx}{\cdot}) f$, where we use $\propto$ to indicate proportionality.

We now recall some basic facts about the LLT, due to \cite{GopiLLST25}. The LLT $\psi$ at $x$ is the cumulant-generating function of the distribution $\tilt_\vx\exp(-\vhi)$, which means that $\psi$ is infinitely-differentiable and that $\nabla^k\psi$ is the $k$th cumulant tensor of $\tilt_\vx\exp(-\vhi)$.

\begin{fact}[LLT derivatives, Lemma 3, {\cite{GopiLLST25}}]\label{fact:llt-derivatives}
    For any $\vx,\vh\in \R^d$, we have \begin{align*}
        \nabla \psi(\vx) &= \mu(\tilt_\vx\exp(-\vhi)) \defeq \E_{\vy\sim \tilt_\vx\exp(-\vhi)}[\vy],\\
        \nabla^2\psi(\vx) &= \Cov(\tilt_\vx\exp(-\vhi)) \defeq \E_{\vy\sim \tilt_\vx\exp(-\vhi)}\Brack{(\vy-\mu(\tilt_\vx\exp(-\vhi))(\vy-\mu(\tilt_\vx\exp(-\vhi))^\top},\\
         \nabla^3\psi(\vx)[\vh,\vh,\vh] &= \E_{\vy\sim \tilt_\vx\exp(-\vhi)}\Brack{\inprod{\vy-\mu(\tilt_\vx\exp(-\vhi))}{\vh}^3}.
    \end{align*}
\end{fact}

\begin{fact}[Lemmas 4--6, {\cite{GopiLLST25}}]\label{fact:llt-smooth-convex}
    The following hold. 
    \begin{enumerate}
        \item[(1)] If $\varphi$ is convex, then $\psi$ is self-concordant.
        \item[(2)] If $\varphi$ is $L$-smooth with respect to $\norm{\cdot}_*$, then $\psi$ is $\frac{1}{L}$-strongly convex with respect to $\norm{\cdot}$.
        \item[(3)] If $\varphi$ is $\frac1L$-strongly convex with respect to $\norm{\cdot}_*$, then $\psi$ is $L$-smooth with respect to $\norm{\cdot}$.
    \end{enumerate}
\end{fact}

We also have the following useful property of the LLT under convolution. 

\begin{lemma}[LLT convolution]\label{lem:lltconv}
Let $\mu = \exp(-\vhi)$, $\nu = \exp(-\phi)$ be densities, and let $\pi \defeq \mu \ast \nu$ be their convolution. Then letting $\chi \defeq -\log(\pi)$, we have $\chi^\sharp = \vhi^\sharp + \phi^\sharp$.
\end{lemma}
\begin{proof}
This is a straightforward calculation: for all $\vw \in \R^d$,
\begin{align*}
\chi^\sharp(\vw) &= \log\Par{\int\exp\Par{\inprod{\vw}{\vy} - \chi(\vy)} \dd \vy} \\
&= \log\Par{\int\exp\Par{\inprod{\vw}{\vy}} \pi(\vy) \dd \vy} \\
&= \log\Par{\int \Par{\int \exp\Par{\inprod{\vw}{\vz}} \nu(\vz) \exp\Par{\inprod{\vw}{\vy - \vz}} \mu(\vy - \vz) \dd \vz} \dd \vy} \\
&= \log\Par{\int \Par{\int  \exp\Par{\inprod{\vw}{\vy - \vz}} \mu(\vy - \vz) \dd \vy} \exp\Par{\inprod{\vw}{\vz}} \nu(\vz) \dd \vz} \\
&= \log\Par{\int \Par{\int  \exp\Par{\inprod{\vw}{\vy}} \mu(\vy) \dd \vy} \exp\Par{\inprod{\vw}{\vz}} \nu(\vz) \dd \vz} \\
&= \log\Par{\int \Par{\int  \exp\Par{\inprod{\vw}{\vy} - \phi(\vy)} \dd \vy} \exp\Par{\inprod{\vw}{\vz} - \vhi(\vz)} \dd \vz} = \phi^\sharp(\vw) + \vhi^\sharp(\vw).
\end{align*}
The first two lines are definitions, the third expands $\pi = \mu \ast \nu$, the fourth applies Fubini's theorem, the fifth performs a change of variables, and the last again applies definitions.
\end{proof}

Lemma~\ref{lem:lltconv} will prove particularly useful to us when convolving multiple copies of a density with itself. For $\vhi: \R^d \to \R$ such that $\pi = \exp(-\vhi)$ is a density, and $\tau \in \N$, we define $\vhi^{\ast \tau}$ to be such that $\exp(-\vhi^{\ast \tau}) = \pi \ast \ldots \ast \pi$, $\tau$ times in total. Then, repeatedly applying Lemma~\ref{lem:lltconv} gives
\[\Par{\vhi^{\ast \tau}}^\sharp = \tau \psi.\]

Our framework in Section~\ref{sec:functional_stochastic_localization} only relies on Lemma~\ref{lem:lltconv} through the above fact, so we always assume that $\int \exp(-\vhi(\vx)) \dd \vx = 1$ without loss of generality, so that we need not carry around normalizing factor adjustments. This can always be achieved by shifting $\vhi$ by a constant, which does not change any of the densities used by our Markov chains (e.g., Algorithm~\ref{alg:alternate}), nor the starting assumptions of any of our main results (e.g., Theorem~\ref{thm:mixing} and Corollary~\ref{cor:slc_mixing}).

A consequence of Lemma~\ref{lem:lltconv} is that if $\psi$ is an LLT, then $\tau\psi$ is an LLT for any $\tau \in \N$. In Appendix~\ref{app:noninteger}, we show that this restrictions to integers is inherent, by proving the following result.

\begin{restatable}{proposition}{restatenonintegerllt}\label{prop:nonintegerllt}
There exists an unconstrained, convex function $\vhi: \R \to \R$ such that $\tau\psi$ is an LLT for $\tau > 0$ iff $\tau \in \N$, where $\psi \defeq \vhi^\sharp$.
\end{restatable}

\paragraph{Localization processes.} We introduce localization processes, primarily following the expositions of \cite{ChenE22,ShiTZ25}.
\begin{definition}\label{def:loc_process}
    A \emph{localization process} is a measure-valued stochastic process $\Brace{\pi_t}_{t\ge0}$ on a state space $\Omega$, which satisfies the following properties:
    \begin{enumerate}
        \item[(1)] $\pi_t$ is a probability measure of $\Omega$ for all $t$, almost surely.
        \item[(2)] For all $A\subseteq \Omega$, $t \mapsto \pi_t(A)$ is a martingale.
        \item[(3)] For all $A\subseteq \Omega$, $\lim_{t\to\infty}\pi_t(A) \in \Brace{0,1}$, almost surely.
    \end{enumerate}
        A \emph{localization scheme} maps $\pi$, a measure over $\Omega$, to a localization process $\Brace{\pi_t}_{t\ge0}$ with initial condition $\pi_0=\pi$. Further, the \emph{localization dynamics} associated with the localization process and some fixed time $T>0$ is the Markov chain with transition kernel defined by \begin{equation}
        \mathsf{p}_T^\pi(A\mid \omega) = \E\Brack{\frac{\pi_T(\omega)\pi_T(A)}{\pi(\omega)}},\quad for\, all\,\omega\in \Omega,\,A\subseteq\Omega.\label{eq:loc_scheme_kernel}
        \end{equation}
\end{definition}

A discrete-time localization process indexed by time $\tau$ can be viewed as a continuous time localization process that is constant on $[\tau, \tau+1)$.

\begin{fact}
    For any $T>0$, $\pi$ is reversible and stationary for \eqref{eq:loc_scheme_kernel}.
\end{fact}
\begin{proof}
    By property (2) of Definition \ref{def:loc_process}, it holds $\E[\pi_T]=\pi$, which implies for $\pi$-almost every $\omega\in\Omega$, \[
    \E\Brack{\frac{\pi_T(\omega)\pi_T(\Omega)}{\pi(\omega)}} = \E\Brack{\frac{\pi_T(\omega)}{\pi(\omega)}}=1.
    \]
    Thus, $\mathsf{p}_T^\pi(\cdot\mid \omega)$ is a probability density. For every $A,B\subseteq \Omega$, by Fubini's theorem, we have \begin{align*}
        \int_B\mathsf{p}_T^\pi(A\mid \omega)\pi(\dd\omega) &= \int_B \E\Brack{\frac{\pi_T(\omega)\pi_T(A)}{\pi(\omega)}}\pi(\dd\omega) = \E\Brack{\int_B\frac{\pi_T(\omega)\pi_T(A)}{\pi(\omega)}\pi(\dd\omega)}\\
        &= \E\Brack{\pi_T(A)\pi_T(B)} = \int_A\mathsf{p}_T^\pi(B\mid \omega)\pi(\dd\omega),
    \end{align*}
so that the Markov chain is reversible and stationary.
\end{proof}

This immediately implies that every localization scheme induces an unbiased sampling algorithm.

\paragraph{Isoperimetric and covariance inequalities.} We require a few additional definitions, centered around isoperimetric inequalities in a metric defined locally by a convex function $\vhi$.

\begin{definition}[$\vhi$-PI]\label{def:vpi}
Let $\pi \propto \exp(-V)$ for $V: \R^d \to \R$, and let $\vhi: \R^d \to \R$ be convex. We say that $\pi$ satisfies a $\vhi$-Poincar\'e inequality with constant $\alpha$ (abbreviated as $\alpha$-$\vhi$-PI) if, for all smooth functions $f: \R^d \to \R$,
\[\Var_{\pi}[f] \le \frac 1 \alpha \E_\pi\Brack{\norm{\nabla f}_{(\nabla^2 \vhi)^{-1}}^2}.\]
\end{definition}
This is also called a mirror Poincar\'e inequality in the literature if $\vhi$ is a mirror map. If $\vhi = \half\norm{\cdot}_2^2$ (so that $\nabla^2 \vhi$ is the identity), this recovers the standard PI.

\begin{fact}[Brascamp-Lieb]\label{fact:brascamp_lieb}
    If $\pi$ is $\alpha$-strongly log-concave with respect to $\vhi$, i.e., $-\log \pi\succeq \alpha \vhi$, then $\pi$ satisfies $\alpha$-$\vhi$-PI.
\end{fact}
Similar to $\alpha$-$\vhi$-PI, we also respectively abbreviate $\alpha$-$\vhi$-strongly log-concave.

\begin{fact}[Cram\'er-Rao]\label{fact:cramer_rao}
If $\pi \propto \exp(-V)$ for $V: \R^d \to \R$,    $\Cov_{X\sim \pi}(X) \succeq\E_{X\sim \pi}[\nabla^2V(X)]^{-1}$.
\end{fact}

\paragraph{Joint distribution induced by the LLT.}

Using Definition \ref{def:llt} and Lemma \ref{lem:lltconv}, we obtain a joint density corresponding to a target distribution $\pi$.  Fix $\tau\in\N\cup\{0\}$, and define $\tilde{\pi}_{(\tau)}\in \mathcal{P}(\calX\times \calY)$ as
\begin{equation}\label{eq:joint-density}
    \tilde{\pi}_{(\tau)}(\vx, \vy):=\exp\Par{\inprod{\vx}{\vy}-\varphi^{*\tau}(\vy)-\tau\psi(\vx)}\pi(\vx).
\end{equation}
It is immediate that $\pi=\tilde{\pi}_{(\tau)}^{X}$.

The joint density $\tilde{\pi}_{(\tau)}$ will be our primary object of study.  In Section \ref{sec:functional_stochastic_localization}, we show that the process obtained from taking $\tilde{\pi}_{(\tau)}^{X\mid Y=\vy_\tau}$ with $\vy_\tau\sim \tilde{\pi}_{(\tau)}^{Y}$ is a discrete-time localization process, where $\tau$ indexes time.  In Section \ref{sec:gibbs_sampling}, we prove general statements about two-state Gibbs sampling and compare these to the localization-based results from \cite{ChenE22}.  In Section \ref{sec:llt_alternate_sampling}, we combine the two perspectives to prove rapid mixing for the proximal sampler induced by $\tilde{\pi}_{(\tau)}$.

\section{Functional Stochastic Localization}\label{sec:functional_stochastic_localization}

Write $\pi_{\tau}^\vy\defeq \tilde{\pi}_{(\tau)}^{X\mid Y=\vy}$.  We construct a discrete-time stochastic localization process via the following:
\begin{equation}\label{eq:llt-loc-process}
    \vx\sim \pi, \quad \va_1, \dots, \va_\tau\stackrel{\text{i.i.d.}}{\sim}\mathcal{T}_{\vx}\exp(-\varphi), \quad \vy_\tau:=\sum_{i=1}^\tau \va_i, \quad \pi_\tau:=\pi_\tau^{\vy_\tau}=\tilde{\pi}_{(\tau)}^{X\mid Y=\vy_\tau},
\end{equation}
with $\vy_0=\mathbf{0}$, so that $\vy_\tau$ is a noisy observation of $\vx$. Intuitively, as $\tau \to \infty$, we receive more signal about $\tilt_{\vx} \exp(-\vhi)$, so \eqref{eq:llt-loc-process} has a natural interpretation as denoising a sample $\vx \sim \pi$. Indeed, taking $\vhi = \half \norm{\cdot}_2^2$, the process \eqref{eq:llt-loc-process} exactly reduces to the noisy Gaussian channel studied by \cite{AlaouiM22}.

Note that $\va_i\mid \vx\sim \exp\Par{\inprod{\va_i}{\vx}-\varphi(\va_i)-\psi(\vx)}$, so by Lemma \ref{lem:lltconv}, 
\[\vy_\tau\mid \vx \sim \exp\Par{\inprod{\vy_\tau}{\vx}-\varphi^{*\tau}(\vy_\tau)-\tau\psi(\vx)}.\]
Then
\[\tpi_{(\tau)}^Y(\vy)= \int \exp\Par{\inprod{\vy}{\vx}-\varphi^{*\tau}(\vy)-\tau\psi(\vx)}\pi(\dd \vx).\]

As stated, \eqref{eq:llt-loc-process} produces $\vy_\tau$ via $\tau$ draws from a ``shared sample'' $\vx$, which appears to be a history-dependent process. We give a convenient reformulation of \eqref{eq:llt-loc-process} that shows it is Markov.

\begin{lemma}\label{lem:ck-markov}
    The process $\{\vy_\tau\}_{\tau\geq 0}$ is Markov.
\end{lemma}
\begin{proof}
    Given $\vy_{\tau}\sim \tpi_{(\tau)}^Y$, we claim that we can generate $\vy_{\tau+1}$ with marginal density $\tpi_{(\tau+1)}^Y$ as follows: draw $\vz\sim \mathcal{T}_{\vy_\tau}\exp(-\tau\psi)\pi$ and $\vw\sim \mathcal{T}_{\vz}\exp(-\varphi)$, and set $\vy_{\tau+1}=\vy_\tau+\vw$ (Algorithm~\ref{alg:localization}). 
    
    This process is clearly Markov.  It remains to show $\vy_{\tau+1}\sim \tpi_{(\tau+1)}^Y$.  Indeed, consider the joint density over $(\vy_{\tau}, \vz, \vw)$: this is proportional to
    \begin{align*}
        &\Par{\int \exp\Par{\inprod{\vy_\tau}{\vx} -\varphi^{*\tau}(\vy_\tau)-\tau\psi(\vx) }\pi(\dd \vx)}\\
        & \cdot \frac{\exp\Par{ \inprod{\vy_{\tau}}{\vz}-\tau\psi(\vz)  }\pi(\vz)}{\int \exp\Par{ \inprod{\vy_{\tau}}{\vz}-\tau\psi(\vz) }\pi(\dd\vz)}\cdot \exp\Par{\inprod{\vz}{\vw}-\varphi(\vw)-\psi(\vz)}\\
        &\qquad=\exp\Par{\inprod{\vy_\tau+\vw}{\vz}-\varphi^{*\tau}(\vy_{\tau})-\varphi(\vw)-(\tau+1)\psi(\vz)}\pi(\vz)\\
        &\qquad=\exp\Par{\inprod{\vy_{\tau+1}}{\vz}-\varphi^{*\tau}(\vy_{\tau})-\varphi(\vy_{\tau+1}-\vy_{\tau})-(\tau+1)\psi(\vz)}\pi(\vz), 
    \end{align*}
    where the last expression is proportional to the joint density over $(\vy_{\tau}, \vz, \vy_{\tau+1})$ via a change of variables.  Integrating over $\vy_\tau$, we notice that this is in the form of a convolution, so $(\vz, \vy_{\tau+1})$ has joint density equal to
    \[\exp\Par{\inprod{\vy_{\tau+1}}{\vz}-\varphi^{*(\tau+1)}(\vy_{\tau+1})-(\tau+1)\psi(\vz)}\pi(\vz).\]
    Finally, marginalizing over $\vz$, we immediately have that $\vy_{\tau+1}\sim \tpi_{(\tau+1)}^Y$.
\end{proof}

\subsection{Renormalization} \label{ssec:renorm}
It is instructive to demonstrate the evolution of $\vy_\tau$ as a Markov process by explicitly giving the Markov operator. In the remainder of this section, we will use $\lambda\in \N\cup \{0\}$ to denote another time index. Define for all $0 \le \tau \le \lambda$:
\begin{equation}
    \begin{aligned}
        V_\tau(\va) &:= -\log \int \exp\Par{\inprod{\va}{\vz}-\tau\psi(\vz)}\pi(\dd\vz),\\
        \sfP_{\tau, \lambda}f(\va)&:=\exp(V_{\tau}(\va))\int f(\va+\vw)\exp(-V_{\lambda}(\va+\vw)-\varphi^{*(\lambda-\tau)}(\vw))\,\dd\vw.
    \end{aligned}
\end{equation}

We can think of $V_\tau$ as ``renormalized potentials'' and $\sfP_{\tau, \lambda}$ as the ``Polchinski operator'' in the framework of \cite{BauerschmidtBD24} but reversed with respect to time; see also \cite{ShiTZ25}. 

Moreover, we can identify $V_\tau$ with $-(\tau\psi-\log \pi)^{\sharp}$, that is, the negative of the log-Laplace transform of $\tau\psi-\log \pi$.  Note that $V_0(\va)=-\log \int \exp(\inprod{\va}{\vz})\pi(\dd\vz)$, so in particular $V_0(\mathbf{0})=0$.

\begin{proposition}\label{prop:renormalization-main}
$\sfP_{\tau, \lambda}$ defines a discrete-time Markov semigroup.  Moreover, $\sfP_{0, \tau}f(\mathbf{0})=\mathbb{E}_{\tpi_{(\tau)}^Y}[f]$ and
\[\tpi_{(\tau)}^Y(\vy)=\exp\Par{-V_\tau(\vy)-\varphi^{*\tau}(\vy)}.\]
\end{proposition}
\begin{proof}
Using the fact that $\exp\Par{-\varphi^{*0}}=\delta_{\mathbf{0}}$ with respect to test functions, 
\[\sfP_{\tau, \tau}f(\va)=\exp(V_\tau(\va))\int f(\va+\vw)\exp\Par{-V_\tau(\va+\vw)-\varphi^{*0}(\vw)}\,\dd\vw=f(\va).\]
Let $\tau\leq \lambda\leq \sigma$.  First, define
\[
    g(\va):=\sfP_{\lambda, \sigma}f(\va)
    =\exp\Par{V_\lambda(\va)}\int f(\va+\vw)\exp\Par{-V_\sigma(\va+\vw)-\varphi^{*(\sigma-\lambda)}(\vw)}\,\dd\vw.
\]
Then,
\begin{align*}
    \sfP_{\tau, \lambda}g(\va)
    &=\exp\Par{V_\tau(\va)}\int g(\va+\vy)\exp\Par{-V_\lambda(\va+\vy)-\varphi^{*(\lambda-\tau)}(\vy)}\,\dd\vy\\
    &=\exp\Par{V_\tau(\va)}\\
     &\quad  \cdot \int \left[ \exp\Par{V_\lambda(\va+\vy)}\int f(\va+\vy+\vw)\exp\Par{-V_\sigma(\va+\vy+\vw)-\varphi^{*(\sigma-\lambda)}(\vw)}\,\dd\vw  \right]\\
     &\quad \cdot \exp\Par{-V_\lambda(\va+\vy)-\varphi^{*(\lambda-\tau)}(\vy)}\,\dd\vy\\
     &=\exp\Par{V_\tau(\va)}\\
     &\quad \cdot \int\int f(\va+\vy+\vw) \exp\Par{-V_\sigma(\va+\vy+\vw)-\varphi^{*(\sigma-\lambda)}(\vw)-\varphi^{*(\lambda-\tau)}(\vy)}\,\dd\vw\,\dd\vy\\
     &=\exp\Par{V_\tau(\va)}\\
     &\quad \cdot \int\int f(\va+\vz) \exp\Par{-V_\sigma(\va+\vz)-\varphi^{*(\sigma-\lambda)}(\vw)-\varphi^{*(\lambda-\tau)}(\vz-\vw)}\,\dd\vw\,\dd\vz\\
     &=\exp\Par{V_\tau(\va)}\int f(\va+\vz)\exp\Par{-V_\sigma(\va+\vz)-\varphi^{*(\sigma-\tau)}(\vz)}\,\dd\vz\\
     &=\sfP_{\tau, \sigma}f(\va),
\end{align*}
which proves the first statement. For the second statement,
\begin{equation}\label{eq:tpiY-test-fns}
    \begin{aligned}
    \sfP_{0, \tau}f(\mathbf{0})
    &=\exp(V_0(\mathbf{0}))\int f(\vw)\exp\Par{-V_\tau(\vw)-\varphi^{*\tau}(\vw)}\,\dd\vw\\
    &=\int\int  f(\vw)\exp\Par{\inprod{\vw}{\vz}-\varphi^{*\tau}(\vw)-\tau \psi(\vz)}\pi(\dd\vz)\,\dd\vw\\
    &=\int f(\vw)\tpi_{(\tau)}^Y(\dd\vw),
\end{aligned}
\end{equation}
as claimed.  The explicit formula for $\tpi_{(\tau)}^Y$ can be seen from the first line of \eqref{eq:tpiY-test-fns} as the definition of integrating functions $f$ with respect to $\tpi_{(\tau)}^Y$.
\end{proof}

The next lemma describes the density of $\vw$ conditioned on $\vy_\tau$ when $\vw$ is generated as in Lemma \ref{lem:ck-markov}.

\begin{lemma}\label{lem:muk}
    Let $\vw$ be generated as follows: for fixed $\vy$, $\vz\sim \mathcal{T}_{\vy}\exp(-\tau\psi)\pi$, and $\vw\sim \mathcal{T}_{\vz}\exp(-\varphi)$.  Then the  density $\nu^\vy_\tau$ of $\vw$ is
    \[\nu^\vy_\tau(\vw)=\exp(-\varphi(\vw)-V_{\tau+1}(\vy+\vw)+V_\tau(\vy)).\]
\end{lemma}
\begin{proof}
    The density of $\mathcal{T}_{\vy}\exp(-\tau\psi)\pi$ is $\exp\Par{\inprod{\vy}{\vz}-\tau\psi(\vz)+V_{\tau}(\vy)}\pi(\vz)$, and the density of $\mathcal{T}_{\vz}\exp(-\varphi)$ is $\exp\Par{\inprod{\vz}{\vw}-\varphi(\vw)-\psi(\vz)}$.  Therefore, we have
    \begin{align*}
        \nu^\vy_\tau(\vw)
        &=\int \exp\Par{\inprod{\vy}{\vz}-\tau\psi(\vz)+V_{\tau}(\vy)}\pi(\vz)\exp\Par{\inprod{\vz}{\vw}-\varphi(\vw)-\psi(\vz)}\,\dd\vz\\
        &=\int \exp\Par{\inprod{\vy+\vw}{\vz}-(\tau+1)\psi(\vz)-\varphi(\vw)+V_\tau(\vy)}\pi(\dd\vz)\\
        &=\exp(-\varphi(\vw)-V_{\tau+1}(\vy+\vw)+V_\tau(\vy)),
    \end{align*}
    as desired.
\end{proof}

The noise $\vw$ added to $\vy_\tau$ at time $\tau$ thus depends on both $\vy$ and $\tau$ via the renormalized potentials.

\begin{lemma}\label{lem:renorm-potential-derivatives}
    The gradient and Hessian of $V_\tau$ satisfy
    \begin{align*}
        \nabla V_\tau(\va)&=-\nabla (\tau\psi(\va)-\log\pi(\va))^{\sharp}=-\mu(\pi_\tau^\va),\\
        \nabla^2 V_{\tau}(\va)&=-\nabla^2 (\tau\psi(\va)-\log\pi(\va))^{\sharp}=-\Cov(\pi_\tau^\va).
    \end{align*}
\end{lemma}
\begin{proof}
    These are direct calculations similar to Fact \ref{fact:llt-derivatives}.
\end{proof}

\subsection{A localization process}
Using our constructions from Section~\ref{ssec:renorm}, we demonstrate that $\pi_\tau$ is a localization process.

\begin{proposition}\label{prop:localization}
    The measure-valued random process $\{\pi_\tau\}_{\tau\geq 0}$ defined in \eqref{eq:llt-loc-process} is a localization process (Definition \ref{def:loc_process}).
\end{proposition}
\begin{proof}
    We check the conditions of Definition \ref{def:loc_process}. First, (1) holds by definition.  To check (2), one possibility is to use Definition 3.1, \cite{Montanari25}; we have that $\pi_\tau$ is the posterior density of $X$ given $\vy_{\tau}$, which by the Markov property, is equal to the posterior density of $X$ given the entire filtration up to $\tau$, i.e., $\pi_\tau$ is a Doob measure-valued martingale.  Alternatively, to demonstrate (2) directly, first define $\mu_\tau$ to be the density of $\vw$ given $\vy_{\tau}$ in the described Markov chain.  Consider the density of $\pi_{\tau+1}$ conditioned on realizations of $\vy_{\tau}$ and $\vw$.  Then
    \begin{align*}
        \pi_{\tau+1}(\vx)
        &=\exp\Par{\inprod{\vy_{\tau+1}}{\vx}-(\tau+1)\psi(\vx)+V_{\tau+1}(\vy_{\tau+1})}\pi(\vx)\\
        &=\exp\Par{\inprod{\vy_{\tau}+\vw}{\vx}-(\tau+1)\psi(\vx)+V_{\tau+1}(\vy_\tau+\vw)}\pi(\vx)\\
        &=\exp\Par{\inprod{\vy_{\tau}}{\vx}-\tau\psi(\vx)+V_{\tau}(\vy_\tau)}\pi(\vx)\cdot \exp\Par{\inprod{\vw}{\vx}-\psi(\vx)-V_\tau(\vy_\tau)+V_{\tau+1}(\vy_\tau+\vw)}\\
        &=\exp\Par{\inprod{\vw}{\vx}-\psi(\vx)-V_\tau(\vy_\tau)+V_{\tau+1}(\vy_\tau+\vw)}  \pi_\tau(\vx).
    \end{align*}
    Therefore, 
    \[
        \mathbb{E}_{\pi_{\tau+1}}[f\mid \pi_\tau]
        =\int \int  f(\vx) \exp\Par{\inprod{\vw}{\vx}-\psi(\vx)-V_\tau(\vy_\tau)+V_{\tau+1}(\vy_\tau+\vw)}  \pi_\tau(\vx) \nu^{\vy_\tau}_\tau(\vw)\,\dd\vw\,\dd\vx.
    \]
    From Lemma \ref{lem:muk}, we have 
    \[\nu^{\vy_\tau}_\tau(\vw)=\exp(-\varphi(\vw)-V_{\tau+1}(\vy_\tau+\vw)+V_\tau(\vy_\tau)),\]
    and so
    \begin{align*}
        \mathbb{E}_{\pi_{\tau+1}}[f\mid \pi_\tau]
        &=\int \int f(\vx)\pi_{\tau}(\vx) \exp\Par{\inprod{\vw}{\vx}-\varphi(\vw)-\psi(\vx)}\,\dd\vw\,\dd\vx\\
        &=\mathbb{E}_{\pi_\tau}[f],
    \end{align*}
    using the definition of the LLT.
    
    To check (3), first fix $\vx\sim \pi$.  Then by the strong law of large numbers and Fact \ref{fact:llt-derivatives},
    \[\frac{\vy_\tau}{\tau}=\frac{\sum_{i=1}^\tau \va_i}{\tau}\xrightarrow[\tau\to\infty]{}\mathbb{E}_{\va_1\sim \mathcal{T}_{\vx}\exp(-\varphi)}[\va_1]=\nabla \psi(\vx).\]
    In other words, $\vy_\tau$ asymptotically gives perfect information about $\vx$ (since $\Par{\nabla \psi}^{-1}=\nabla \psi^*$, where $\psi^*$ is the convex conjugate of $\psi$; see \cite{BubeckE19}), and $\pi_\tau$ thus converges to $\delta_{\vx}, \vx\sim \pi$, as we have realized $\pi_\tau$ as a posterior density.
\end{proof}

We summarized the procedure for generating samples from $\pi_{\tau}$ in Algorithm \ref{alg:localization}.  Note that the algorithm relies on oracles for sampling from $\mathcal{T}_{\vy}\exp(-j\psi)\pi$ (for all $0 \le j \le \tau - 1$) and $\mathcal{T}_{\vz}\exp(-\vhi)$.
\section{Gibbs Sampling}\label{sec:gibbs_sampling}
In this section, we revisit general two-state Gibbs sampling, and provide several structural facts relating the variance contraction of the ``primal process'' that samples from one of two state spaces, to the variance contraction of the corresponding ``dual process'' on the other state space. These facts will be used to analyze Algorithm~\ref{alg:localization} in Section~\ref{sec:llt_alternate_sampling}. We remark much of the content in this section has been derived in the information theory literature through the lens of strong data processing inequalities \cite{Raginsky16,PolyanskiyW25}; we include brief rederivations to keep our exposition self-contained.

For readers more familiar with discrete probability, we provide an intuitive explanation of our main characterizations in Lemmas~\ref{lem:twice_gap} and~\ref{lem:equivalence-forward-backward} when both spaces are discrete in Appendix~\ref{app:gibbs_sampling_discrete}.

In this section, we work with two spaces $\calX, \calY$, equipped with $\sigma$-algebras $\Sigma_\calX,\Sigma_\calY$, respectively. For intuition, the reader may assume $\calX=\calY=\R^d$ equipped with the Borel $\sigma$-algebra.  $X, Y$ denote random variables on $\calX, \calY$, respectively.  Let $\Omega=\calX\times\calY$ be equipped with the product $\sigma$-algebra, and as stated in Section \ref{sec:preliminaries}, let $\tilde{\pi}$ be a joint
probability measure on $\Omega$ with marginal densities $\tilde{\pi}^X$, $\tilde{\pi}^Y$ and conditional densities $\tilde{\pi}^{Y\mid X}$, $\tilde{\pi}^{X\mid Y}$.

Consider the Hilbert spaces $L^2(\tpi^X), L^2(\tpi^Y)$, and define the operator $\sfK:L^2(\tpi^Y)\to L^2(\tpi^X)$ as
\begin{equation}\label{eq:forward-operator}
    [\sfK g](\vx)\defeq \int_{\calY} g(\vy)\tpi^{Y\mid X=\vx}(\dd\vy)=\mathbb{E}[g(Y)\mid X=\vx].
\end{equation}
We can derive the adjoint $\sfK^\dagger:L^2(\tpi^X)\to L^2(\tpi^Y)$ as follows
\begin{align*}
    \inprod{\sfK g}{f}_{L^2(\tpi^X)}
    &=\int_{\calX} f(\vx)[\sfK g](\vx)\tpi^{X}(\dd\vx)\\
    &=\int_{\calX} \int_{\calY} f(\vx)g(\vy)\tpi^{Y\mid X=\vx}(\dd\vy)\tpi^X(\dd \vx)\\
    &=\int_{\calX} \int_{\calY} f(\vx)g(\vy) \tpi(\dd\vx, \dd\vy)\\
    &=\int_{\calX}\int_{\calY} f(\vx)g(\vy)\tpi^{X\mid Y=\vy}(\dd\vx)\tpi^Y(\dd \vy)\\
    &=: \inprod{g}{\sfK^\dagger f}_{L^2(\tpi^Y)},
\end{align*}
which is simply Bayes' rule.  Then,
\begin{equation}\label{eq:backward-operator}
    [\sfK^\dagger f](\vy)\defeq \int_{\calX}f(\vx)\tpi^{X\mid Y=\vy}(\dd\vx)=\mathbb{E}[f(X)\mid Y=\vy].
\end{equation}

A single iteration of Gibbs sampling starts from $\vx\in\calX$, samples $\vy\sim\tpi^{Y\mid X=\vx}$, and finally samples
$\vx'\sim\tpi^{X\mid Y=\vy}$. Symmetrically, there are analogous dynamics on $\calY$. These dynamics reflectively define Markov operators $\sfP_{\calX}\defeq \sfK\sfK^\dagger, \sfP_{\calY}\defeq \sfK^\dagger \sfK$.  

It is immediate that $\tpi^X$ is stationary for $\sfP_{\calX}$. Indeed, 
\begin{align*}
    \mathbb{E}_{\tpi^X}[\sfP_{\calX} f]
    &= \int_{\calX }[\sfP_{\calX} f](\vx)\tpi^X(\dd\vx)\\
    &= \int_{\calX }[\sfK \sfK^\dagger f](\vx)\tpi^X(\dd\vx)\\
    &=\int_{\calX }\int_{\calX}\int_{\calY} f(\vx') \tpi(\vx, \dd\vy)\tpi(\dd\vx', \vy)\tpi^X(\dd\vx)\\
    &=\int_{\calX }\int_{\calX}\int_{\calY} f(\vx') \tpi(\dd\vx, \dd\vy)\tpi(\dd\vx', \vy)\\
    &=\int_{\calX}\int_{\calY}f(\vx')\tpi(\dd\vx', \vy)\tpi^Y(\dd\vy) \\
    &=\int_{\calX} f(\vx')\tpi^X(\dd\vx')\\
    &=\mathbb{E}_{\tpi^X}[f],
\end{align*}
and a similar calculation holds for $\sfP_{\calY}$.  Moreover, $\sfP_{\calX}, \sfP_{\calY}$ are self-adjoint on $L^2(\tpi^X), L^2(\tpi^Y)$:
\begin{equation}\label{eq:reversible}
    \inprod{\sfP_{\calX}f_1}{f_2}_{L^2(\tpi^X)}=\inprod{\sfK \sfK^\dagger f_1}{f_2}_{L^2(\tpi^X)}=\inprod{\sfK^\dagger f_1}{\sfK^\dagger f_2}_{L^2(\tpi^Y)}
    =\inprod{f_1}{\sfP_{\calX}f_2}_{L^2(\tpi^X)},
\end{equation}
and a similar calculation holds for $\sfP_{\calY}$. 

\subsection{Variance contraction via spectral gap}
Now, we interlude with a discussion on the mixing rate for Markov chains associated to \emph{general} localization processes by completing the argument in Proposition 19, \cite{ChenE22}. That is, Proposition 19, \cite{ChenE22} gives a generic way of converting a variance conservation bound in a localization scheme to a spectral gap on the associated Markov chain. We provide a more explicit variant of this calculation that in fact holds for generic Markov chains, not just Gibbs samplers. For an analogous derivation in the discrete space setting, see Example 33.8, \cite{PolyanskiyW25}.

\begin{lemma}\label{lem:twice_gap}
For all functions $f \in L^2(\tpi^X)$,
\[\frac{\Var_{\tpi^X}[\sfP_{\calX} f]}{\Var_{\tpi^X}[f]} \le \lam_2\Par{\sfP_{\calX}}^2.\]
\end{lemma}
\begin{proof}
We have
\begin{align*}
    \operatorname{Var}_{\tpi^X}[\sfP_{\calX}f]
    &=\mathbb{E}_{\tpi^X}[(\sfP_{\calX}f-\mathbb{E}_{\tpi^X}[\sfP_\calX f])^2]\\
    &=\mathbb{E}_{\tpi^X}[(\sfP_\calX f)^2]-\mathbb{E}_{\tpi^X}[f]^2\\
    &=\mathbb{E}_{\tpi^X}[f^2]-\mathbb{E}_{\tpi^X}[f^2]+\mathbb{E}_{\tpi^X}[(\sfP_\calX f)^2]-\mathbb{E}_{\tpi^X}[f]^2\\
    &=\operatorname{Var}_{\tpi^X}[f]-\Par{\mathbb{E}_{\tpi^X}[f^2]-\mathbb{E}_{\tpi^X}[(\sfP_\calX f)^2]}\\
    &=\operatorname{Var}_{\tpi^X}[f]-\Par{\norm{f}_{\tpi^X}^2-\inprod{\sfP_\calX f}{\sfP_\calX f}_{\tpi^X}}\\
    &=\operatorname{Var}_{\tpi^X}[f]-\Par{\norm{f}_{\tpi^X}^2-\inprod{\sfP_\calX^2 f}{f}_{\tpi^X}}
\end{align*}
by reversibility.  This object in the parentheses is the Dirichlet form $\mathcal{E}(f, f)$ associated with $\sfP_\calX^2$.  Then 
\begin{align*}
    \operatorname{Var}_{\tpi^X}[f]-\Par{\norm{f}_{\tpi^X}^2-\inprod{\sfP_\calX^2 f}{f}_{\tpi^X}}
    &=\Par{1-\frac{\mathcal{E}(f, f)}{\operatorname{Var}_{\tpi^X}[f]}}\operatorname{Var}_{\tpi^X}[f]\\
    &\leq \lambda_2(\sfP_\calX^2)\operatorname{Var}_{\tpi^X}[f]\\
    &=\lambda_2(\sfP_\calX)^2\operatorname{Var}_{\tpi^X}[f],
\end{align*}
obtaining the squared rate.
\end{proof}

We apply Lemma~\ref{lem:twice_gap} to prove a convergence rate on Gibbs sampling from the LLT-induced joint density \eqref{eq:joint-density} in Section~\ref{sec:llt_alternate_sampling}. Interestingly, for variance contraction specifically, the specialization of our result to Euclidean stochastic localization yields a rate that is twice as fast as derived in existing applications of localization schemes machinery (e.g., Theorem 58, \cite{ChenE22}), which owes to our sharper characterization in Lemma~\ref{lem:twice_gap}. We remark that the same generic ``forwards contraction implies backwards contraction'' characterization of Gibbs samplers is false in full generality, in the case of entropy instead of variance, cf.\ \cite{CaputoCGP25} for a counterexample.

\subsection{Contraction rate for Gibbs sampling} 

We now give a more precise characterization of this squared rate in Gibbs sampling due to contraction of both forward and backward channels (at the same rate). In particular, we only require proving a bound on the contraction factor of one of the two channels, as contraction of the other channel follows from the fact that the operators on the two channels are adjoint.

We start with the following two lemmas.

\begin{lemma}\label{lem:mean-zero-map}
    The operators $\sfK, \sfK^\dagger$ map mean-zero functions to mean-zero functions, i.e., for $g\in L^2(\tpi^Y)$ with  $\mathbb{E}_{\tpi^Y}[g]=0$ and  $f\in L^2(\tpi^X)$ with $\mathbb{E}_{\tpi^X}[f]=0$, 
    \begin{align*}
        &\mathbb{E}_{\tpi^X}[\sfK g]=0,\\
        &\mathbb{E}_{\tpi^Y}[\sfK^\dagger f]=0.
    \end{align*}
\end{lemma}
\begin{proof}
    This is a simple calculation in iterated expectation.  We have
    \[\mathbb{E}_{\tpi^X}[\sfK g]=\mathbb{E}_{\tpi^X}[\mathbb{E}[g(Y)\mid X=\vx]]=\mathbb{E}_{\tpi^Y}[g]=0.\]
    A similar calculation proves the statement for $\sfK^\dagger$.
\end{proof}

\begin{lemma}\label{lem:norm-subspaces}
    Write $\sfK_{\mid \mathbf{1}^\perp}, \sfK^\dagger_{\mid \mathbf{1}^\perp}$ to be the restriction of the operators $\sfK, \sfK^\dagger$ to mean-zero functions on $L^2(\tpi^Y), L^2(\tpi^X)$, respectively.  Then $\norms{\sfK_{\mid \mathbf{1}^\perp}}_{L^2(\tpi^Y)}=\norms{\sfK^\dagger_{\mid \mathbf{1}^\perp}}_{L^2(\tpi^X)}$.
\end{lemma}
\begin{proof}
    It suffices to show $\sfK^\dagger_{\mid \mathbf{1}^\perp}=(\sfK_{\mid \mathbf{1}^\perp})^\dagger$. The result then follows as the norms of a bounded operator and its adjoint are equal.  We have, for $f, g$ mean-zero in $L^2(\tpi^X), L^2(\tpi^Y)$, respectively,
    \begin{align*}
        \inprod{g}{\Par{\Par{\sfK_{\mid \mathbf{1}^\perp}}^\dagger -\sfK^\dagger_{\mid \mathbf{1}^\perp} } f}_{L^2(\tpi^Y)}
        &=\inprod{g}{\Par{\sfK_{\mid \mathbf{1}^\perp}}^\dagger f}_{L^2(\tpi^Y)} -\inprod{g}{\sfK^\dagger_{\mid \mathbf{1}^\perp} f}_{L^2(\tpi^Y)}\\
        &=\inprod{g}{\Par{\sfK_{\mid \mathbf{1}^\perp}}^\dagger f}_{L^2(\tpi^Y)} -\inprod{g}{\sfK^\dagger f}_{L^2(\tpi^Y)}\\
        &=\inprod{\sfK_{\mid \mathbf{1}^\perp} g}{f}_{L^2(\tpi^X)}-\inprod{\sfK g}{f}_{L^2(\tpi^X)}\\
        &=0.
    \end{align*}
    By non-degeneracy of the inner product, we have the claim $(\sfK_{\mid \mathbf{1}^\perp})^\dagger =\sfK^\dagger_{\mid \mathbf{1}^\perp}$.
\end{proof}

Our interpretation of variance contraction for Gibbs sampling is captured by the following lemma. 

\begin{lemma}\label{lem:equivalence-forward-backward}
    We have \begin{equation}\label{eq:primal_dual_contract}
    \sup_{\mathbb{E}_{\tpi^Y}[g]=0}\frac{\Var_{\tpi^X}[\sfK g]}{\Var_{\tpi^Y}[g]} = \sup_{\mathbb{E}_{\tpi^X}[f]=0}\frac{\Var_{\tpi^Y}[\sfK^\dagger f]}{\Var_{\tpi^X}[f]}.
    \end{equation}
Moreover, if both quantities in \eqref{eq:primal_dual_contract} are $\beta < 1$, then for all mean-zero $f$ in $L^2(\tpi^X)$,
\[\frac{\operatorname{Var}_{\tpi^X}[\sfP_\calX f]}{\operatorname{Var}_{\tpi^X}[f]} \leq \beta^2 .\]
\end{lemma}
\begin{proof}
The first statement is a consequence of Lemmas~\ref{lem:mean-zero-map} and~\ref{lem:norm-subspaces}. To see the second, for any $f$ with $\E_{\tpi^X}[f]=0$, we have \begin{align*}
        \Var_{\tpi^X}[\sfP_\calX f] =\Var_{\tpi^X}[\sfK \sfK^\dagger f] \le \beta \Var_{\tpi^Y}[ \sfK^\dagger f] \le \beta^2\Var_{\tpi^X}[f],
    \end{align*}
    as long as we ensure $\sfK^\dagger f$ is mean-zero; but this is Lemma \ref{lem:mean-zero-map}.
\end{proof}

Lemma~\ref{lem:equivalence-forward-backward} is a somewhat more explicit variant of Lemma~\ref{lem:twice_gap}, which shows that controlling just one of the two quantities in \eqref{eq:primal_dual_contract} is sufficient to obtain the squared contraction rate for the overall Gibbs sampler. This is convenient in Section~\ref{sec:llt_alternate_sampling}, as one of the relevant bounds in \eqref{eq:primal_dual_contract} is much more straightforward to obtain in our particular application of LLT proximal sampling.

\section{LLT Proximal Sampling}\label{sec:llt_alternate_sampling}

We instantiate Gibbs sampling from the previous section with the LLT-based proximal sampler using the joint density \eqref{eq:joint-density}, restated for convenience:
\[\tilde{\pi}_{(\tau)}(\vx, \vy)=\exp\Par{\inprod{\vx}{\vy}-\varphi^{*\tau}(\vy)-\tau\psi(\vx)}\pi(\vx),\]
where $\pi$ is the target distribution and $\tau$ is a fixed step-size.  With respect to this joint distribution, let $\sfK, \sfK^\dagger$ be the operators as in \eqref{eq:forward-operator}, \eqref{eq:backward-operator}, respectively.  We can see that the corresponding Gibbs sampler is implemented as Algorithm~\ref{alg:alternate} (note that this implementation of the process uses the perspective in \eqref{eq:llt-loc-process}, rather than its Markov reformulation in Lemma~\ref{lem:ck-markov}).

\begin{algorithm2e}[t!]
\caption{$\mathsf{LLT\text{-}Prox}(\tau, \vx_0, K)$}
\label{alg:alternate}
\DontPrintSemicolon
\textbf{Input: }localization parameter $\tau \in \N$, initial $\vx_0 \in \R^d$, iteration count $K \in \N$\\
\For{$0 \le k < K$}{
$\vy_k \sim \tilt_{\vx_k}\exp(-\vhi^{*\tau})$\; \label{line:sample_y}
$\vx_{k + 1} \sim \tilt_{\vy_k}\exp(-\tau\psi)\pi$\; \label{line:sample_x}
}
\textbf{Return:} $\vx_K$
\end{algorithm2e}

We first show the dynamics corresponding to Gibbs sampling on $\tilde{\pi}_{(\tau)}$ are equivalent to the localization dynamics $\mathsf{p}_\tau^\pi$ \eqref{eq:loc_scheme_kernel} induced by the localization process of \eqref{eq:llt-loc-process}.
\begin{proposition}\label{prop:gibbs_loc}
    The dynamics corresponding to Gibbs sampling on $\tilde{\pi}_{(\tau)}$ are equivalent to the localization dynamics induced by $\{\pi_\tau\}_{\tau\geq 0}$ defined in \eqref{eq:llt-loc-process}.
\end{proposition}
\begin{proof}
    From the construction given in \eqref{eq:llt-loc-process}, we draw $\vw\sim \pi$ and $\vy\mid \vw\sim \tilde{\pi}_{(\tau)}^{Y\mid X=\vw}$, and set $\pi_\tau=\tilde{\pi}_{(\tau)}^{X\mid Y=\vy}$.  Then
    \begin{align*}
        \mathsf{p}_{\tau}^\pi(\vx'\mid \vx) 
        &=\mathbb{E}\left[ \frac{\pi_\tau(\vx')\pi_\tau(\vx)}{\pi(\vx)} \right]\\
        &=\mathbb{E}_{(\vw, \vc)\sim \tilde{\pi}_{(\tau)}}\left[ \frac{\tilde{\pi}_{(\tau)}^{X\mid Y=\vy}(\vx')\tilde{\pi}_{(\tau)}^{X\mid Y=\vy}(\vx)}{\pi(\vx)} \right]\\
        &=\mathbb{E}_{\vc\sim \tilde{\pi}_{(\tau)}^Y}\left[ \frac{\tilde{\pi}_{(\tau)}^{X\mid Y=\vy}(\vx')\tilde{\pi}_{(\tau)}(\vx, \vy)}{\tilde{\pi}_{(\tau)}^Y(\vy)\pi(\vx)} \right]\\
        &=\int \tilde{\pi}_{(\tau)}^{X\mid Y=\vy}(\vx') \tilde{\pi}_{(\tau)}^{Y\mid X=\vx}(\vy)\,\dd\vy.
    \end{align*}
    This is precisely Gibbs sampling.
\end{proof}

Henceforth, as we fix $\tau$ as step-size (in contrast to Section \ref{sec:functional_stochastic_localization} and Proposition~\ref{prop:gibbs_loc}, where $\tau$ is viewed as time), we write $\tpi:=\tpi_{(\tau)}$.  We now prove our main variance contraction result.

\begin{lemma}\label{lem:dual-poincare}
    Let $\vhi$ be convex, $\psi\defeq \vhi^\sharp$, and let $\pi=\tpi^X$ satisfy an $\alpha$-$\psi$-Poincar\'e inequality, and let $(X,Y)\sim \tpi$. For every smooth $f:\calY\to\R$, \begin{equation}\label{eq:halfway_var}
        \Var_{\tpi^X}[(\sfK f)(X)]\le \frac{1}{1+\alpha/\tau}\Var_{\tpi^Y}[f(Y)].
    \end{equation}
\end{lemma}
\begin{proof}
Write $g(\vx) \defeq \mathbb{E}_{\vy\sim \tilde{\pi}^{Y\mid X=\vx}}[f(\vy)]$ as the conditional mean. We first claim that
\begin{equation}\label{eq:var_compare}
\Var_{\vx \sim \tpi^X}[g(\vx)] \le \frac \tau \alpha \E_{\vx \sim \tpi^X}[\Var_{\tpi^{Y \mid X = \vx}}[f]].
\end{equation}
We start by introducing the following helpful notation, where we use the derivations from Fact~\ref{fact:llt-derivatives}:
\begin{align*}
\vm(\vx) &\defeq \E_{\vy \sim \tilde{\pi}^{Y\mid X=\vx}}[\vy] = \nabla \psi(\vx), \\
\msig(\vx) &\defeq \Cov_{\vy \sim \tilde{\pi}^{Y\mid X=\vx}}[\vy] = \tau\nabla^2 \psi(\vx), \\
f_\vx(\vy) &\defeq f(\vy) - \E_{\vy' \sim \tilde{\pi}^{Y\mid X=\vx}}[f(\vy')]= f(\vy) - g(\vx), \\
\vy_\vx &\defeq \vy - \E_{\vy'\sim\tilde{\pi}^{Y\mid X=\vx}}[\vy']= \vy - \vm(\vx).
\end{align*}
In other words, $f_\vx$ and $\vy_\vx$ are ``centered'' versions of $f$ and $\vy$. Then we have
\begin{align*}
\nabla g(\vx) &= \nabla_{\vx}\Par{ \int f(\vy)\exp\Par{\inprod{\vx}{\vy} - \vhi^{*\tau}(\vy) - \tau \psi(\vx)} \dd \vy} \\
&= \int \Par{\vy - \nabla \psi(\vx)} f(\vy)\exp\Par{\inprod{\vx}{\vy} - \vhi^{*\tau}(\vy) - \tau\psi(\vx)} \dd \vy \\
&= \int \vy_{\vx} f(\vy)\exp\Par{\inprod{\vx}{\vy} - \vhi^{*\tau}(\vy) - \tau\psi(\vx)} \dd \vy \\
&= \int \vy_{\vx} f_{\vx}(\vy)\exp\Par{\inprod{\vx}{\vy} - \vhi^{*\tau}(\vy) - \tau\psi(\vx)} \dd \vy \\
&= \E_{\vy \sim \tilde{\pi}^{Y\mid X=\vx}}[\vy_\vx f_\vx(\vy)].
\end{align*}
Next, observe that for any vector $\vv \in \R^d$, 
\begin{align*}
\Par{\vv^\top \nabla g(\vx)}^2 &= \Par{\E_{\vy \sim \tilde{\pi}^{Y\mid X=\vx}}\Brack{\vv^\top\vy_\vx f_\vx(\vy)}}^2 \\
&\le \E_{\vy \sim \tilde{\pi}^{Y\mid X=\vx}}\Brack{f_\vx(\vy)^2} \cdot \E_{\vy \sim \tilde{\pi}^{Y\mid X=\vx}}\Brack{(\vv^\top \vy_\vx)^2} \\
&= \Var_{\tilde{\pi}^{Y\mid X=\vx}}[f] \cdot \vv^\top \msig(\vx) \vv.
\end{align*}
Plugging in $\vv = \msig(\vx)^{-1} \nabla g(\vx)$ gives
\begin{equation}\label{eq:cond_grad_bound}
\nabla g(\vx)^\top \msig(\vx)^{-1} \nabla g(\vx) \le \Var_{\tilde{\pi}^{Y\mid X=\vx}}[f].
\end{equation}
Now if $\pi=\tpi^X$ satisfies $\alpha$-$\psi$-PI, then it also satisfies $\frac \alpha \tau$-$\tau\psi$-PI.  Then \eqref{eq:var_compare} follows:
\[
\Var_{\vx \sim \tpi^X}[g(\vx)] \le \frac{\tau}{\alpha} \E_{\vx \sim \tpi^X}\Brack{\nabla g(\vx)^\top \msig(\vx)^{-1} \nabla g(\vx)} \\
\le \frac{\tau}{\alpha} \E_{\vx \sim \tpi^X}\Brack{\Var_{\tpi^{Y\mid X=\vx}}[f]}, \]
where the second inequality used \eqref{eq:cond_grad_bound}. Finally, we obtain \eqref{eq:halfway_var} using the law of total variance \eqref{eq:total-var}:
\[\Var_{\tpi^Y}[f(Y)] = \underbrace{\E_{\tpi^X}[\Var[f(Y)\mid X]]}_A +\underbrace{ \Var_{\tpi^X}[\E[f(Y)\mid X]]}_B.\]
Indeed, \eqref{eq:var_compare} gives $B \le \frac \tau \alpha A$, and rearranging gives the claim.
\end{proof}

We mention that the derivation in Lemma~\ref{lem:dual-poincare} shares significant similarity with proofs of the van Trees inequality (Section 2.4.2, \cite{van2004detection}), particularly in considering the joint covariance of $f(\vy), \vy$. Moreover, when specialized to the Euclidean setting, it uses a similar proof strategy as Theorem 3.2, \cite{Srinivasan25}.

\begin{lemma}\label{lem:norm-bounds}
    Let $f\in L^2(\tpi^X), g\in L^2(\tpi^Y)$ be mean-zero functions.  Then if $\pi=\tpi^X$ satisfies a $\psi$-Poincar\'e inequality with constant $\alpha$,
    \begin{align*}
        & \norm{\sfK g}_{L^2(\tpi^X)}\leq \frac{1}{\sqrt{1+\alpha/\tau}}\norm{g}_{L^2(\tpi^Y)},\\
        & \norm{\sfK^\dagger f}_{L^2(\tpi^Y)}\leq \frac{1}{\sqrt{1+\alpha/\tau}}\norm{f}_{L^2(\tpi^X)}.
    \end{align*}
\end{lemma}
\begin{proof}
    The first statement follows from Lemma~\ref{lem:dual-poincare}.  The second follows from Lemma \ref{lem:norm-subspaces}.
\end{proof}

We are now ready to prove our main result on the convergence of Algorithm~\ref{alg:alternate}.

\restatemixing*
\begin{proof}
    Let $h_k \defeq \frac{\dd\mu_k}{\dd\tpi^X}-1$ so that $\chi^2(\mu_0\| \tpi^X)=\norm{h_0}^2_{L^2(\pi^X)}$.  From \eqref{eq:reversible} we have $h_k=\sfP_{\calX}^k h_0$, so that $\chi^2(\mu_k\|\tpi^X) =\norm{h_k}^2_{L^2(\tpi^X)}= \norm{\sfP_\calX^kh_0}_{L^2(\tpi^X)}^2$.  As 
    \[\norm{\sfP_\calX h_k}^2_{L^2(\tpi^X)}\leq \frac{1}{(1+\alpha/\tau)^2}\norm{h_{k-1}}_{L^2(\tpi^X)}^2\]
    by Lemmas~\ref{lem:equivalence-forward-backward} and \ref{lem:norm-bounds}, recursing and noting $h_0$ is mean-zero furnishes the result.
\end{proof}

As a corollary, we achieve the quadratic improvement to the main result of \cite{GopiLLST25}.
\begin{corollary}\label{cor:slc_mixing}
    Let $\pi = \tpi^X\propto \exp(-V)$ such that $V\succeq \alpha\psi$.  Then Theorem \ref{thm:mixing} holds for $\pi$.
\end{corollary}
\begin{proof}
    The statement is an application of Brascamp-Lieb (Fact \ref{fact:brascamp_lieb}).
\end{proof}

For completeness, we prove a variant of Theorem~\ref{thm:mixing} in Appendix~\ref{app:proof_via_conductance} using machinery closer to that in \cite{GopiLLST25}, i.e., classical conductance arguments based on isoperimetry. Our proof in Appendix~\ref{app:proof_via_conductance} yields a looser constant, and also is less flexible, obtaining an asymptotically worse rate when $\alpha \le 1$. This discrepancy shows the advantage of leveraging our direct characterization of variance contraction, as opposed to isoperimetric properties afforded by self-concordance of LLTs.

\section{Zeroth-Order Private Convex Optimization}\label{ssec:dp}

In this section, we apply our framework to problems in differentially private convex optimization. To keep our presentation self-contained, we first define differential privacy \cite{DworkMNS06,DworkR14}. Let $\calS$ be an abstract domain, and let $n \in \N$. We say that a mechanism (randomized algorithm) $\calM: \calS^n \to \Omega$ that takes size-$n$ datasets in $\calS^n$ to outcomes in $\Omega$ satisfies $(\eps, \delta)$-differential privacy (DP) if for any event $S \subseteq \Omega$, and any two datasets $\calD, \calD' \in \calS^n$ that differ in exactly one element,
\[\Pr\Brack{\calM(\calD) \in S} \le \exp(\eps)\Pr\Brack{\calM(\calD') \in S} + \delta.\]
We now define the main pair of problems we study.

\begin{problem}[DP-ERM and DP-SCO]\label{prob:dpco}
Let $n \in \N$, $\eps, \delta \in (0, 1)$, $D, G \ge 0$, and let $\xset \subset \R^d$ be compact and convex with diameter at most $D$ in a norm $\norm{\cdot}$. Let $\calP$ be a distribution over a set $\calS$ such that for any $\vs \in \calS$, there is a $f(\cdot; \vs): \xset \to \R$ which is convex and $G$-Lipschitz in $\norm{\cdot}$. Let $\calD \defeq \{\vs_i\}_{i \in [n]}$ consist of $n$ independent draws from $\calP$, and let $f_i \defeq f(\cdot; \vs_i)$ for all $i \in [n]$. 

In the \emph{differentially private empirical risk minimization (DP-ERM)} problem, we receive $\calD$ and wish to design a mechanism $\calM$ which satisfies $(\eps, \delta)$-DP and approximately minimizes
\[\Ferm(\vx) \defeq \frac 1 n \sum_{i \in [n]} f_i(\vx).\]
In the \emph{differentially private stochastic convex optimization (DP-SCO)} problem, we receive $\calD$ and wish to design a mechanism $\calM$ which satisfies $(\eps, \delta)$-DP and approximately minimizes
\[\Fpop(\vx) \defeq \E_{\vs \sim \calP}\Brack{f(\vx; \vs)}.\]
\end{problem}

Our goal is to obtain improved algorithms for Problem~\ref{prob:dpco}, when the norm $\norm{\cdot}$ used to measure the regularity of the instance is an $\ell_p$ norm, for $p \in [1, 2)$. We focus on solving Problem~\ref{prob:dpco} under a \emph{zeroth-order (value) oracle} access model, where we can query the value of $f_i(\vx)$ for any $i \in [n]$ and $\vx \in \calX$. While Problem~\ref{prob:dpco} has also been studied under first-order oracle access \cite{AsiFKT21, BassilyGN21}, all first-order methods currently lose logarithmic factors (or worse) in their excess risk bounds. To address this gap, \cite{GopLL22, GopiLLST23a} introduced a sampling approach that leverages zeroth-order oracles within a regularized exponential mechanism to obtain improved tradeoffs.

While these previous sampling-based approaches \cite{GopiLLST23a, GopiLLST25} establish state-of-the-art tradeoffs between the sample size $n$ and the excess risk, the resulting algorithms are inefficient from the perspective of the number of oracle queries used. Specifically, \cite{GopiLLST23a} used the proximal sampler of \cite{LeeST21}, which was catered to Euclidean geometry and thus lost dimension-dependent factors in its oracle query complexity for $p < 2$. On the other hand, \cite{GopiLLST25} developed a basic non-Euclidean proximal sampler using the theory of LLTs, that yielded a suboptimal convergence analysis compared to Euclidean counterparts. This gave a quadratic overhead in the oracle query complexity of the \cite{GopiLLST25} algorithm, even when granted a warm start. Our main application uses our functional stochastic localization framework to remove the quadratic overhead of the \cite{GopiLLST25} algorithm, while retaining its state-of-the-art excess risk bounds.

As in prior results, our main workhorse is the following structural fact, showing that sampling from an appropriately-regularized distribution provides the stability required for DP to hold.

\begin{proposition}[Theorem 3, Theorem 4, \cite{GopiLLST23a}, Theorem 6.9, \cite{GopLL22}]\label{prop:gennormdp}
In the setting of Problem~\ref{prob:dpco}, let $k \ge 0$, and let $r: \xset \to \R$ be $1$-strongly convex with respect to $\norm{\cdot}$, with additive range $\max_{\vx \in \xset} r(\vx) - \min_{\vx \in \xset} r(\vx)$ at most $\Theta$. Let $\pi$ be the density on $\xset$ satisfying $\dd \pi(\vx) \propto \exp(-k(\Ferm(\vx) + \mu r(\vx))) \dd \vx$. Then the algorithm which returns a sample from $\pi$ for
\begin{equation}\label{eq:erm_params}k =\frac{\sqrt d n\eps}{G\sqrt{2\Theta \log \frac 1 {2\delta}}},\;  \mu = \frac{2G^2k\log \frac 1 {2\delta}}{n^2\eps^2},\end{equation}
satisfies $(\eps, \delta)$-DP, and guarantees
\begin{align*}
\E_{\vx \sim \pi}\Brack{\Ferm(\vx)} - \min_{\vx \in \xset}\Ferm(\vx) \le O\Par{G\sqrt{\Theta} \cdot \frac{\sqrt{d \log \frac 1 \delta}}{n\eps}}.
\end{align*}
Further, the algorithm which returns a sample from $\pi$ for
\begin{equation}\label{eq:sco_params}k = \frac{1}{G\sqrt{\Theta}} \cdot \sqrt{\Par{\frac{d\log \frac 1 {2\delta}}{\eps^2 n^2} + \frac{1}{n}}} \cdot \min\Par{\frac{\eps^2 n^2}{\log \frac 1 {2\delta}}, nd},\; \mu = G^2 k \cdot \max\Par{\frac{\log \frac 1 {2\delta}}{n^2\eps^2}, \frac{1}{nd}}\end{equation}
satisfies $(\eps, \delta)$-DP, and guarantees
\[\E_{\vx \sim \pi}\Brack{\Fpop(\vx)} - \min_{\vx \in \xset}\Fpop(\vx) \le O\Par{G\sqrt{\Theta} \cdot \Par{\frac{\sqrt{d\log \frac 1 \delta}}{n\eps} + \frac 1 {\sqrt n}}}.\]
\end{proposition}

Proposition~\ref{prop:gennormdp} gives a straightforward way to attack Problem~\ref{prob:dpco}. Our goal is to design a regularizer $r: \calX \to \R$ that is $1$-strongly convex in a norm $\norm{\cdot}$ with small additive range $\Theta$, and such that
\begin{equation}\label{eq:dp_density}\dd\pi(\vx) \propto \exp\Par{-k\Ferm(\vx) - k\mu r(\vx)}\dd \vx\end{equation}
admits an efficient sampler from the perspective of zeroth-order query complexity, for the choices of parameters $k, \mu$ in \eqref{eq:erm_params}, \eqref{eq:sco_params}. Our strategy is to directly choose $r$ to be a multiple of an appropriate LLT, and to apply our framework to sample from the density \eqref{eq:dp_density} efficiently.

\subsection{Setup}\label{ssec:dp_setup}

To obtain our sampling result, we recall some preliminaries from \cite{GopiLLST25}. 

\paragraph{Regularizer construction.} 
To construct an appropriate $r$ for use within Proposition~\ref{prop:gennormdp}, we need several helper facts. Recall the second item in Fact~\ref{fact:llt-smooth-convex}, which shows that to obtain strong convexity when $\norm{\cdot}_*$ is an $\ell_p$ norm for $p \in [1, 2)$, we should take the LLT of a function smooth in the dual $\ell_q$ norm, where $\frac 1 p + \frac 1 q = 1$. We now provide such a function and bound the range of its LLT.

\begin{fact}[\cite{BallCL94, KakadeST09}]\label{fact:dual_smooth}
For all $p \in [1, 2)$ and $a > 0$, letting $q > 2$ satisfy $\frac 1 p + \frac 1 q = 1$, the function 
\begin{equation}\label{eq:regdualdef}
\vhi_{p, a}(\vy) \defeq a\norm{\vy}_q^2
\end{equation}
is $\frac {2a} {p - 1}$-smooth with respect to $\norm{\cdot}_q$ over $\R^d$.
\end{fact}

\begin{fact}[Lemma 19, \cite{GopiLLST25}]\label{fact:llt_range}
For all $p \in [1, 2)$ and $a > 0$, letting $q > 2$ satisfy $\frac 1 p + \frac 1 q = 1$, define
\begin{equation}\label{eq:regdef}
\begin{aligned}
\psi_{p, a}(\vx) \defeq \vhi_{p, a}^\sharp(\vx) = \log\Par{\int \exp\Par{\inprod{\vx}{\vy} - a\norm{\vy}_q^2} \dd \vy}.
\end{aligned}
\end{equation}
Then the additive range of $\psi_{p, a}$ over the unit $\ell_p$ ball $\{\vx \in \R^d \mid \norm{\vx}_p \le 1\}$ is
\[O\Par{1 + \frac 1 a + \sqrt{\frac d a \log\Par{a + \frac d a }}}.\]
\end{fact}

\paragraph{Conditional samplers.} Our algorithm applies functional stochastic localization to the target density \eqref{eq:dp_density}, by choosing $k\mu r = \alpha \psi_{p, a}$ as defined in \eqref{eq:regdualdef} for appropriate choices of $a, \alpha$. Denoting $V \defeq k\Ferm$, this results in alternate sampling from a joint density of the form
\begin{equation}\label{eq:dp_joint}\exp\Par{\inprod{\vx}{\vy} - \vhi_{p, a}^{\ast \tau}(\vy) - \tau \psi_{p, a}(\vx)} \pi(\vx), \text{ where } \pi(\vx) \propto \exp\Par{-V(\vx) - \psi_{p, a}(\vx)}.\end{equation}

We hence additionally require the existence of oracle query-efficient and time-efficient samplers for sampling from the induced conditionals of 
\eqref{eq:dp_joint}. For the $\vx$-conditional, we use the following result.

\begin{proposition}[Proposition 1, \cite{GopiLLST25}]\label{prop:reject}
Let $\delta \in (0, 1)$, and let $\nu \propto \exp(-U(\vx) - R(\vx))$ be a density over compact, convex $\calX \subseteq \R^d$, where $R: \calX \to \R$ is $\mu$-strongly log-concave in $\norm{\cdot}$, and $U = \E_{i \simu \calI} U_i$ for some index set $\calI$ and $\{U_i: \calX \to \R\}_{i \in \calI}$ that are all $G$-Lipschitz in $\norm{\cdot}$. If
\[\mu \ge 10^4 G^2 \log\Par{\frac 1 \delta},\]
there is an algorithm that queries (in expectation) $O(1)$ value oracles for $U_i$ where $i \simu \calI$, and $O(1)$ samples drawn from the density $\propto \exp(-R)$ within total variation distance $\Omega(\delta)$, and outputs a sample within total variation distance $\delta$ from $\nu$.
\end{proposition}
Finally, both our ``base distribution sampler'' (for sampling $\propto \exp(-R)$ in uses of Proposition~\ref{prop:reject}), and our sampler for the $\vy$-conditional of \eqref{eq:dp_joint}, leverage state-of-the-art generic samplers for log-concave densities \cite{KookV25a}. The former also requires an efficient evaluation oracle for $\psi_{p,a}$.

For the following statement, we work in a standard computational model where, given value access to a one-dimensional function $f: \R \to \R_{\ge 0}$, we assume we can exactly evaluate the integral $\int_{x \in \R} f(x) \dd x$ and sample from the density $\propto f$ in $O(1)$ time. This assumption is standard and holds up well in practice for most well-behaved $f$; accounting for the imprecision of numerical solvers adds polylogarithmic overhead in problem parameters, and is omitted to avoid tedium.

\begin{fact}[Section 5.3, \cite{GopiLLST25}]\label{fact:llt_eval}
For $p \in [1, 2)$ and $a > 0$, there is a value oracle for $\psi_{p, a}$ that is implementable in $O(d)$ time, and we can sample from $\tilt_{\vx} \exp(-\vhi_{p, a})$ for any $\vx \in \R^d$ in $O(d)$ time.
\end{fact}

\begin{proposition}[Theorem 1.7, \cite{KookV25a}]\label{prop:logcon_sample}
Let $\delta \in (0, 1)$, let $\nu \propto \exp(-V)$ for convex $V: \calX \to \R$ and compact, convex $\calX \subset \R^d$ with diameter at most $R > 0$ in $\norm{\cdot}_2$, and let $r > 0$, $\vx_0 \in \calX$ be such that all $\vx \in \calX$ with $\norm{\vx - \vx_0}_2 \le r$ satisfy $V(\vx) - \min_{\vx \in \calX} V(\vx) \le d$. There is an algorithm that uses
\[O\Par{d^{3.5} \polylog\Par{\frac{dR}{\delta r}}}\]
value oracle queries to $V$, and outputs a sample within total variation distance $\delta$ to $\nu$.
\end{proposition}

\subsection{Main application}

We conclude by putting together the tools from Section~\ref{ssec:dp_setup} to produce a joint density of the form \eqref{eq:dp_joint} that is efficiently sampleable within our framework. To begin we require an assumption.

\begin{assumption}\label{assume:warmstart}
Fix $p \in [1, 2)$ and $k, a > 0$. Assume there is an algorithm $\calA_{\textup{warmstart}}$ that returns a sample $\sim \hat{\pi}$, satisfying $\chi^2(\hat{\pi} \| \pi) \le \beta$ for some $\beta \ge 0$,
where $\pi$ is the density in \eqref{eq:dp_density} for $k\mu r = \psi_{p, a}$.
\end{assumption}

In the setting of Problem~\ref{prob:dpco}, Assumption~\ref{assume:warmstart} holds unconditionally using no oracle queries, with $\beta = \exp(O(GD))$. This follows from boundedness of $\calX$ and Proposition~\ref{prop:logcon_sample} (cf.\ Lemma 18, \cite{GopiLLST25}). Applying this constructive warm start within our framework would lose a polynomial in problem parameters, similarly to \cite{GopiLLST25}. We leave the explicit construction of such a warm start with $\beta$ depending polynomially on parameters to future work. Existence of an entropic variant of Theorem~\ref{thm:mixing}, as conjectured in Appendix~\ref{app:entropy}, would eliminate the need for Assumption~\ref{assume:warmstart}, as then our mixing times would depend doubly-logarithmically on $\beta$.

\begin{corollary}\label{cor:dp_sample}
Let $p \in [1, 2)$ and $\eps, \delta \in (0, 1)$. In the setting of Problem~\ref{prob:dpco}, where $\norm{\cdot}$ is the $\ell_p$ norm on $\R^d$, there is an $(\eps, \delta)$-DP algorithm $\calM_{\textup{erm}}$ that outputs $\vx \in \calX$ satisfying
\begin{align*}
\E_{\vx \sim \mecherm}\Brack{\Ferm(\vx)} - \min_{\vx \in \calX} \Ferm(\vx) &= O\Par{\frac{GD}{\sqrt{p - 1}} \cdot \frac{\sqrt{d\log \frac 1 \delta}}{n\eps}} \text{ for } p \in (1, 2), \\
\E_{\vx \sim \mecherm}\Brack{\Ferm(\vx)} - \min_{\vx \in \calX} \Ferm(\vx) &= O\Par{GD\sqrt{\log d} \cdot \frac{\sqrt{d\log \frac 1 \delta}}{n\eps}} \text{ for } p = 1.
\end{align*}
Further, there is an $(\eps, \delta)$-DP algorithm $\mechsco$ that outputs $\vx \in \calX$ satisfying
\begin{align*}
    \E_{\vx \sim \mechsco}\Brack{\Ferm(\vx)} - \min_{\vx \in \calX} \Ferm(\vx) &= O\Par{\frac{GD}{\sqrt{p - 1}} \cdot \Par{\frac 1 {\sqrt n} + \frac{\sqrt{d\log \frac 1 \delta}}{n\eps}} } \text{ for } p \in (1, 2), \\
\E_{\vx \sim \mechsco}\Brack{\Ferm(\vx)} - \min_{\vx \in \calX} \Ferm(\vx) &= O\Par{GD\sqrt{\log d} \cdot \Par{\frac 1 {\sqrt n} + \frac{\sqrt{d\log \frac 1 \delta}}{n\eps}}} \text{ for } p = 1.
\end{align*}
Both $\mecherm$ and $\mechsco$ call $\calA$ in Assumption~\ref{assume:warmstart}, appropriately parameterized, once. $\mecherm$ uses
\[O\Par{\Par{1 + n^2\eps^2\log(d)\log\Par{\frac{nd\log\beta}{\delta}}}\log\Par{\frac \beta \delta}}\]
additional value queries (in expectation) to some $f(\cdot; \vs_i)$, and $\mechsco$ uses
\[O\Par{\Par{1 + \min\Par{n^2\eps^2, nd}\log(d)\log\Par{\frac{nd\log\beta}{\delta} }}\log\Par{\frac \beta \delta}}\]
such additional value queries (in expectation). Both algorithms use $\poly(d, \log(\frac{n\beta}{\delta}))$ additional time.
\end{corollary}
\begin{proof}
As in Theorem 2, \cite{GopiLLST25}, we make the following simplifying assumptions: we let $D = 1$, $p - 1 = \Omega(\frac 1 {\log d})$, and allow for $\delta$ total variation error in the application of Proposition~\ref{prop:gennormdp}. The former two assumptions are correct by scale invariance of the problem and norm distortions in $\R^d$. The last follows from a union bound and standard facts about approximate DP (Lemma 3.17, \cite{DworkR14}). 

Under these simplifications, we take $k\mu r = \alpha \psi_{p, a}$ in \eqref{eq:dp_density} for $k, \mu$ given by \eqref{eq:erm_params} or \eqref{eq:sco_params}, and appropriate choices of $\alpha, a$ so that $r$ is $1$-strongly convex in $\norm{\cdot}_p$. By Facts~\ref{fact:llt-smooth-convex} and~\ref{fact:dual_smooth}, we require
\[\frac{\alpha(p - 1)}{2a} \ge k\mu \iff \alpha \ge \frac{2ak\mu}{p-1}. \]
We choose $a = \frac{1}{d\log(d)}$ and $\alpha = \frac{2ak\mu}{p-1}$. Applying Fact~\ref{fact:llt_range} for $r \defeq \frac{\alpha}{k\mu} \psi_{p, a}$ shows that it suffices to take the additive range of $r$ to be
\[\Theta = O\Par{\frac{\alpha}{ak\mu}} = O\Par{\frac{1}{p-1}}.\]
With these choices in hand, we begin by discussing our DP-ERM algorithm, specializing $k, \mu$ to their definitions in \eqref{eq:erm_params}. Let $V \defeq k \Ferm = \E_{i \simu [n]} kf_i$, so that $V$ is a uniform mixture over $kG$-Lipschitz functions. We apply Theorem~\ref{thm:mixing} for $T$ iterations, to the density
$\pi \propto \exp(-V - \alpha \psi_{p, a})$, 
for a parameter $\tau$ such that $\tau \alpha \psi_{p, a}$ meets the strong convexity requirement in Proposition~\ref{prop:reject}. For failure probability $O(\frac \delta T)$ in Proposition~\ref{prop:reject}, this requires
\[\tau k\mu = \Omega\Par{k^2 G^2\log\Par{\frac T \delta}} \iff \tau = \Omega\Par{\frac{k G^2 \log(\frac T \delta)}{\mu}}.\]
Finally, by plugging in the parameters from \eqref{eq:erm_params}, we may choose (for appropriate constants)
\[\tau = O\Par{n^2 \eps^2 \log\Par{\frac{nd\log\beta}{\delta}}},\quad \alpha = \Omega\Par{\frac{1}{\log d}},\quad T = O\Par{\frac \tau \alpha \log\Par{\frac \beta \delta}}\]
for Theorem~\ref{thm:mixing} to give $O(\delta)$ total variation distance to $\pi$ under Assumption~\ref{assume:warmstart}.
The query complexity follows because Proposition~\ref{prop:reject} uses $O(1)$ value oracle queries in expectation per iteration of Theorem~\ref{thm:mixing}, and all other sampling access is provided in polynomial time by Fact~\ref{fact:llt_eval} and Proposition~\ref{prop:logcon_sample}.

The analysis of our DP-SCO algorithm is identical, except we use the parameter settings in \eqref{eq:sco_params} instead. These settings yield the same bounds on $\alpha$ and $T$, and
\[\tau = O\Par{\min\Par{n^2 \eps^2, nd} \log\Par{\frac{nd\log\beta}{\delta}} }.\]
This concludes the proof.
\end{proof}

Both parts of our framework (the definition of the localization process in Section~\ref{sec:functional_stochastic_localization}, and our analysis technique in Section~\ref{sec:llt_alternate_sampling}) were critical in obtaining Corollary~\ref{cor:dp_sample} without additional overheads in the query complexity. As discussed after Theorem~\ref{thm:mixing_alt}, the \cite{GopiLLST25} variant of the LLT framework only permits taking $\tau = 1$, and has a quadratically worse dependence on the scale-invariant parameter $\frac \tau \alpha$. This led to roughly a quadratically-worse query complexity than Corollary~\ref{cor:dp_sample}. 

On the other hand, if we were solely to change the sampling framework to that in Section~\ref{sec:functional_stochastic_localization}, but keep the analysis the same, we would obtain Theorem~\ref{thm:mixing_alt} as our mixing time result. Because $\alpha = o(1)$ is required in the parameter settings of Corollary~\ref{cor:dp_sample} to obtain the tight excess risk, this would again lead to an extraneous $\frac 1 \alpha$ factor in the query complexity, that is removable using Theorem~\ref{thm:mixing}.
\section{Conclusion}\label{sec:conclusion}

Our work significantly expands upon the theory and potential applications of proximal sampling via the log-Laplace transform by constructing a novel localization process and connecting it to two-state Gibbs sampling. There are several important open directions, which we believe will be of interest to the community.

\paragraph{Faster mixing rate.} While we improve upon the quadratic dependence on the relative convexity parameter from \cite{GopiLLST25} and further weaken this assumption to a Poincar\'e inequality, our mixing rate still scales linearly in $\log \beta$, where $\beta$ is the warm-start parameter, which may incur dimension-dependent factors, e.g., see the discussion before Corollary~\ref{cor:dp_sample}. This is avoided in the Euclidean setting by using the stronger condition of a log-Sobolev inequality and conducting the analysis in KL divergence, which achieves a runtime scaling in $\log\log \beta$. 
Although we take an initial step towards this result in Appendix~\ref{app:entropy} under Assumption~\ref{assume:llt}, it remains to prove the result in broader generality. Specifically, it is interesting to understand whether a weaker condition than Assumption~\ref{assume:llt} suffices, as well as the types of regularizers $\psi \defeq \vhi^\sharp$ that satisfy such conditions.

\paragraph{Samplers for explicit functions.} Our work primarily  focuses on establishing the theory of functional stochastic localization and its connection to the non-Euclidean proximal sampler from \cite{GopiLLST25}, leaving sharper, practical, algorithmic guarantees of sampling of the conditional distributions for future work. We briefly mention that \cite{GopiLLST25} provides some algorithms for $\ell_p$ geometries, e.g., via rejection sampling and the inverse Laplace transform, as discussed in Section~\ref{ssec:dp}. However, even these applications require generic log-concave sampling algorithms (in Proposition~\ref{prop:reject}), which do not exploit problem structure. This could potentially be avoided for specific $\vhi$ of interest.

\paragraph{Additional applications.} Our main application in Section~\ref{ssec:dp} is for sampling in continuous spaces. We are hopeful that our framework will have applications for sampling in discrete spaces as well; see for instance \cite{AnariHLVXY23}, where the theory of LLTs was used in a discrete sampling problem. 

Subsequent work \cite{GuSTY26} utilizes our framework for polytope sampling under the observation that the logarithmic barrier is an LLT. This work also addresses the faster mixing rate under cold-starts by providing a functional cooling framework inspired by \cite{KookV24}; its rate is governed by the self-concordance parameter of $\psi$.

\section*{Acknowledgments}
We thank Matthew S.\ Zhang for helpful discussions during the initial stages of this project. KT thanks Sinho Chewi for a helpful conversation regarding the approach in Appendix~\ref{app:entropy}. We also thank Hank Green for pointing out an error in an intermediate claim made in a previous version of this work. This error is discussed in Appendix~\ref{sec:erratum}, and does not affect any of our main results.

\newpage
\bibliographystyle{alpha}
\bibliography{refs}

\appendix
\section{Noninteger LLT Counterexample}\label{app:noninteger}

In this appendix, we prove Proposition~\ref{prop:nonintegerllt}.

\restatenonintegerllt*
\begin{proof}
The ``if'' direction follows from Lemma~\ref{lem:lltconv}, so let $\tau > 0$ be noninteger henceforth. Let
\[
h(y) = \frac{1}{2\sqrt{2\pi}}\int_{-1}^1\exp\Par{-\frac{(y-u)^2}{2}}\dd u
\]    
and let $\vhi\defeq -\log h$. Note that $h$ is the convolution of the log-concave random variables $U\sim \mathsf{Unif}([-1,1])$ and $Z\sim \calN(0, 1)$, so $\vhi$ is finite, smooth, and convex. Let the Laplace transform be \[
\Psi(x) = \int_\R \exp\Par{xy-\vhi(y)}\dd y = \exp\Par{\frac{x^2}{2}}\frac{\sinh x}{x},\footnote{We use the convention $\frac{\sinh x} x = 1$ at $x = 0$; after removing this singularity, $\frac{\sinh x} x$ is entire.}
\]
so that the LLT is $\psi(x) = \log \Psi(x) = \frac{x^2}{2}+\log \frac{\sinh x}{x}$. Note that $\Psi$ is an entire function (see Proposition 3.2.2, \cite{Duchi23}; this also follows by direct computation and dominated convergence).

Now, suppose $\tau \psi$ is an LLT for $\tau \not\in \N$, so that $\tau\psi(x) = \log(\int_{\R} \exp(xy)\mu(y) \dd y)$ for some measure $\mu$ over $\R$. Then, let $\Psi_\tau(z) \defeq \int_{\R} \exp(zy) \mu(y) \dd y$, which is also entire. 
Define the entire function $H(z) \defeq \Psi'_\tau(z) \Psi(z) - \tau\Psi_\tau(z)\Psi'(z)$. Since $\Psi(x) > 0$ for all $x \in \R$, we may differentiate the  identity $\log \Psi_\tau = \tau \log \Psi$ on $\R$, to conclude that $H(x) = 0$ for all $x \in \R$. Thus, by the identity theorem for analytic functions (Chapter 4, Section 3.2, \cite{Ahlfors79}), $H \equiv 0$ over $\C$, so meromorphically,
\begin{equation}
\frac{\Psi_\tau'(z)}{\Psi_\tau(z)} = \tau \frac{\Psi'(z)}{\Psi(z)}  .\label{eq:meromorphic}
\end{equation}

We next observe that the left-hand side of \eqref{eq:meromorphic} has integer residues. This is shown in Chapter 4, Section 5.2, \cite{Ahlfors79}, and follows because residues are the $-1^{\text{th}}$ term in the Laurent expansion, and if $\Psi_\tau$ is analytic with an order-$m$ zero $a \in \C$, then $\Psi_\tau = (\cdot - a)^m f$ for some analytic $f$ with $f(a) \neq 0$.

We finally claim that the right-hand side of \eqref{eq:meromorphic} has a residue of $\tau$ at some pole. To see this, letting $s(z) \defeq \frac{\sinh z}{z}$ with $s(0) = 1$, it is equivalent to show that $1$ is a residue of the logarithmic derivative
\[\frac{\Psi'(z)}{\Psi(z)} = z + \frac{s'(z)}{s(z)}. \]
Indeed, the argument of Chapter 4, Section 5.2 reduces this to verifying that $z = i\pi$ is an order-$1$ zero of $s(z)$, which is a simple calculation. Thus we have reached a contradiction, because the left-hand side of \eqref{eq:meromorphic} cannot have a residue of $\tau$ at $z = i\pi$.
\end{proof}
\section{Gibbs Sampling for Discrete Spaces}\label{app:gibbs_sampling_discrete}

In this section, we perform a derivation that clarifies the spectral gap of the Gibbs sampler $\sfP_\calX$ of Section \ref{sec:gibbs_sampling} in the discrete setting, where $\calX = [n]$ and $\calY = [m]$.

Let $\mpp\in\R_{\ge0}^{n\times m}$ be the matrix corresponding to the joint probabilities of $\tpi$, so that $\mpp_{xy}=\tpi(x, y)$. Then $\vr = \mpp\mathbf{1}_m, \vc = \mpp^\top \mathbf{1}_n$ describe the marginal probabilities, and form the matrices $\mr = \diag{\vr}, \mc = \diag{\vc}$. A direct calculation shows that the conditional probabilities are: \[
\tpi^{Y|X=x}(y) = \frac{\mpp_{xy}}{\vr_x},\qquad \tpi^{X|Y=y}(x) = \frac{\mpp_{xy}}{\vc_y}.
\]
Gibbs sampling as described in Section \ref{sec:gibbs_sampling} thus yields the transition matrices
\[
\sfK=\mr^{-1}\mpp, \quad \sfK^\dagger = \mc^{-1}\mpp^\top, \quad \sfP_{\calX}=\mr^{-1}\mpp\mc^{-1}\mpp^\top.
\]
We proceed by deriving the spectral gap of $\sfP_\calX$.  We have
\[
    \frac{1}{2}\sum_{x,x'}\vr_x(\sfP_\calX)_{xx'}(\vvhi_x-\vvhi_{x'})^2\\
    =\vvhi^\top \mathbf{R}\vvhi-\vvhi^\top \mpp \mc^{-1}\mpp^\top \vvhi.
\]
Then
\begin{align*}
    \gap(\sfP_\calX)
    &=\inf_{\vvhi\in \mathbb{R}^n}\frac{\sum_{x,x'}\vr_x(\sfP_\calX)_{xx'}(\vvhi_x-\vvhi_{x'})^2}{2\operatorname{Var}_{\tpi^\calX}[\vvhi]}\\
    &=\inf_{\vvhi\in \mathbb{R}^n}\frac{\vvhi^\top \mathbf{R}\vvhi-\vvhi^\top \mpp \mc^{-1}\mpp^\top \vvhi}{\vvhi^\top (\mr-\vr\vr^\top)\vvhi}\\
    &= 1 - \sup_{\vvhi\in \mathbb{R}^n}\Par{1-\frac{\vvhi^\top \mathbf{R}\vvhi-\vvhi^\top \mpp \mc^{-1}\mpp^\top \vvhi}{\vvhi^\top (\mr-\vr\vr^\top)\vvhi}}\\
    &=1 - \sup_{\vvhi\in \mathbb{R}^n}\Par{\frac{ \vvhi\mpp \mc^{-1}\mpp^\top \vvhi-\vvhi \vr\vr^\top\vvhi}{\vvhi^\top (\mr-\vr\vr^\top)\vvhi}}\\
    &=1-\sup_{\vvhi\in \mathbb{R}^n} \frac{\operatorname{Var}_{\tpi^Y}[\sfK^\dagger \vvhi]}{\operatorname{Var}_{\tpi^X}[\vvhi]}.
\end{align*}
This calculation, essentially the contents of Lemma~\ref{lem:equivalence-forward-backward} in the discrete case, shows that the spectral gap of $\sfP_{\calX}$ does not characterize the full variance contraction of the forward-backward Gibbs sampling process, and instead captures the contraction through just one of the two channels. By performing the same calculation as above in the other direction, we conclude that $\gap(\sfP_{\calX})^2$ gives the full variance contraction of Gibbs sampling, as in Lemma~\ref{lem:twice_gap}.

\section{Proof of Theorem~\ref{thm:mixing} Weakening via Conductance}\label{app:proof_via_conductance}

In this section, we give a brief presentation of how to apply the mixing time estimate technique in \cite{GopiLLST25} to establish a convergence bound for Algorithm~\ref{alg:alternate}.  However, as we demonstrate, this proof requires additional assumptions, yielding a weaker guarantee in certain parameter ranges (closer to that of \cite{GopiLLST25}, i.e., with a quadratic overhead) and worse constants overall.

Before we prove the main result, we will state some preliminaries. We say a convex $\vhi:\R^d\to\R$ is self-concordant if \[
|\nabla^3\vhi(\vx)[\vh,\vh,\vh]|\le 2(\nabla^2\vhi(\vx)[\vh,\vh])^{\frac32}, \quad\text{for all }\vx,\vh\in\R^d.
\]

We will use the following three facts to prove a bound on the Poincar\'e constant. 
\begin{fact}[Lemmas 4 and 12, \cite{ChenDWY20}]\label{fact:old_school_pi_bound}
Consider a Markov chain on $\R^d$ with transition kernels $\{\mathsf{t}_{\vx}\}_{\vx\in\mathbb{R}^d}$ and stationary distribution $\pi$. Suppose the following two conditions hold for some metric $\textup{m}: \R^d \times \R^d \to \R_{\ge 0}$.
\begin{enumerate}
    \item For all $\vx, \vx' \in \R^d$ with $\textup{m}(\vx, \vx') \le \Delta$, we have $D_{\textup{TV}}(\mathsf{t}_{\vx}, \mathsf{t}_{\vx'}) \le \frac 1 2$.
    \item The density $\pi$ satisfies $\alpha$-isoperimetry in $\textup{m}$, i.e., for all partitions $S_1 \cup S_2 \cup S_3 = \R^d$,
    \[\frac{\pi(S_3)}{\min\Brace{\pi(S_1), \pi(S_2)}} \ge \alpha \inf_{\vx \in S_1, \vx' \in S_2}\textup{m}(\vx, \vx').\]
\end{enumerate}
Then the Poincar\'e constant ($\chi^2$ contraction) of the Markov chain is $\Omega((\Delta \alpha)^2)$.
\end{fact}

We note that \cite{ChenDWY20} gives a more modern statement of Fact~\ref{fact:old_school_pi_bound} that depends on the entire \emph{isoperimetric profile} of $\pi$, at different scales of expansion.

\begin{fact}[Lemma 9, \cite{GopiLLST25}]\label{fact:lem9_gllst}
    Let $\psi:\R^d$ be convex and self-concordant, and let  $\pi$ be $\alpha$-$\psi$-strongly log-concave for $\alpha\in(0,1]$. For any partition $S_1\cup S_2\cup S_3$ of $\R^d$, \[
    \frac{\pi(S_3)}{\min\Brace{\pi(S_1), \pi(S_2)}} = \Omega\Par{\alpha\inf_{\vx\in S_1,\vx'\in S_2}\textup{m}_\psi(\vx,\vx')}.
    \]
\end{fact}

\begin{fact}[Lemma 10, \cite{GopiLLST25}]\label{fact:lem10_gllst}
Let $\vhi: \R^d \to \R$ be convex and $\psi \defeq \vhi^\sharp$ be an LLT. For any $\vx,\vx'\in \R^d$ such that $\textup{m}_\psi(\vx,\vx')\le \frac14$, \[
    \dtv(\tilt_\vx\exp(-\vhi),\tilt_{\vx'}\exp(-\vhi))\le\frac{1}{2}.
    \]
\end{fact}
In the above two results, $\textup{m}_\psi$ is defined to be the metric induced by $\nabla^2\psi$. We defer a more complete discussion of the Riemannian geometry involved to \cite{GopiLLST25}.

Now to prove our mixing time bounds, we bound the Poincar\'e constant.
\begin{lemma}\label{lem:pi_constant_by_conductance}
Let $\vhi: \R^d \to \R$ be convex, $\psi \defeq \vhi^\sharp$, and let $V$ be $\alpha$-$\psi$-strongly convex.  Then the Poincar\'e constant of one step of Algorithm \ref{alg:alternate} with $\pi\propto \exp(-V)$ as target distribution is $\Omega(\frac{\min\{\alpha, \alpha^2\}}\tau )$.
\end{lemma}
\begin{proof}
We lower bound the Poincar\'e constant in the two cases $\alpha \ge 1$ and $\alpha \le 1$.

First, suppose that $\alpha \ge 1$, and define $\omega \gets \alpha \psi$. Then, $\omega$ is self-concordant, so Fact~\ref{fact:lem9_gllst} concludes that for $\textup{m} \gets \textup{m}_{\omega}$, we may take $\alpha = 1$ in Fact~\ref{fact:old_school_pi_bound}. Further, taking $\Delta = \frac 1 4 \sqrt{\alpha/\tau}$, we have
\[\textup{m}_{\omega}\Par{\vx, \vx'} \le \Delta \implies \textup{m}_{\tau \psi}\Par{\vx, \vx'} \le \frac 1 4,\]
so that Fact~\ref{fact:lem10_gllst} with $\vhi \gets \vhi^{\ast \tau}$ now gives
\[
\dtv(\mathsf{t}_\vx, \mathsf{t}_{\vx'})\leq \dtv(\tilt_\vx\exp(-\vhi^{*\tau}),\tilt_{\vx'}\exp(-\vhi^{*\tau}))\le\frac{1}{2},
\]
where the first inequality uses the coupling characterization of the TV distance. Plugging in these choices of $\alpha$ and $\Delta$ into Fact~\ref{fact:old_school_pi_bound} now gives the claimed lower bound.

In the other case $\alpha \le 1$, we directly apply Fact~\ref{fact:lem9_gllst} with $\textup{m} \gets \textup{m}_\psi$, so that $\alpha \gets \alpha$ in Fact~\ref{fact:old_school_pi_bound}. Then, the analogous application of Fact~\ref{fact:lem10_gllst} requires $\Delta = \frac 1 4 \sqrt{1/\tau}$, and then Fact~\ref{fact:old_school_pi_bound} gives the result.
\end{proof}

Equipped with this lemma, we prove an analogous result to Theorem~\ref{thm:mixing}.
\begin{theorem}\label{thm:mixing_alt}
Let $\vhi: \R^d \to \R$ be convex, $\psi \defeq \vhi^\sharp$, and $\pi=\tpi^X \propto \exp(-V)$ such that $V \succeq \alpha\psi$.  For any $\mu_0\ll \pi$, iterate $\vx_k$ of the proximal sampler (Algorithm \ref{alg:alternate}) has density $\mu_k$ satisfying
    \[
        \chi^2(\mu_k \|\pi) \le \frac{1}{(1 + O(\min\{\alpha, \alpha^2\}/\tau))^{k}}\chi^2(\mu_0\|\pi).
    \]
\end{theorem}
\begin{proof}
    This is a direct consequence of Lemma~\ref{lem:pi_constant_by_conductance}.
\end{proof}

We briefly compare Theorem~\ref{thm:mixing_alt} to Theorem~\ref{thm:mixing}, as it differs in three main ways. First, the explicit constant hidden in the rate is weaker, as it relies on the characterization in Fact~\ref{fact:old_school_pi_bound} which is typically lossy. Second, as the isoperimetric inequality used in its proof is based on log-concavity (via Fact~\ref{fact:lem9_gllst}), this proof technique must rely on $\psi$-strong convexity as its starting assumption, rather than a weaker assumption such as a $\psi$-Poincar\'e inequality (as in Theorem~\ref{thm:mixing}).

Third, and perhaps qualitatively most importantly, Theorem~\ref{thm:mixing_alt} is unable to match the $\approx \frac \tau \alpha$ convergence rate of Theorem~\ref{thm:mixing} unless $\alpha \ge 1$. This led to a particularly stark gap in the framework of \cite{GopiLLST25}, which only permitted taking $\tau = 1$, and thus their analog of Theorem~\ref{thm:mixing_alt} was quadratically worse than Theorem~\ref{thm:mixing}. Our localization framework in Section~\ref{sec:functional_stochastic_localization} allows for taking an arbitrary $\tau \in \N$, and our analysis strategy in Section~\ref{sec:functional_stochastic_localization} also achieves the correct dependence on $\alpha$ (Theorem~\ref{thm:mixing}), by bypassing the prior work's reliance on self-concordant isoperimetry.
\section{Entropic Contraction Under Assumption~\ref{assume:llt}}\label{app:entropy}
In this section, we prove entropic contraction of Algorithm~\ref{alg:alternate}
under a condition that we assume $\psi$, the LLT used, satisfies (Assumption~\ref{assume:llt}). While we cannot prove Assumption~\ref{assume:llt} in full generality (and indeed, numerically there appear to be counterexamples), it is a natural condition that we believe serves as a stepping stone towards more generic proofs of entropic contraction.

We start with some definitions. For a convex function $\vhi:\R^d\to\R$, the induced Bregman divergence and Wasserstein distance are respectively defined as
\begin{align*}
    D^\vhi(\vx'\|\vx) &\defeq \vhi(\vx') - \vhi(\vx) -\inprod{\nabla\vhi(\vx)}{\vx'-\vx},\\
    W^\vhi(\mu'\|\mu) &\defeq \inf_{\gamma \in \Gamma(\mu',\mu)}\E_{(\vx,\vx')\sim \gamma}\Brack{D^\vhi(\vx'\|\vx)}
\end{align*}
where $\Gamma(\mu',\mu)$ is the set of all couplings of the two densities $\mu',\mu$, i.e., joint densities whose marginals are as specified. We will also use the following two helper facts from prior work.

\begin{fact}[Lemma 10, \cite{GopiLLST25}]\label{fact:kld_breg}
For a density $\mu =\exp(-\vhi): \R^d \to \R_{\ge 0}$ and tilts $\vx, \vx' \in \R^d$,
\[\dkl\Par{\tilt_{\vx'}\mu \| \tilt_{\vx}\mu} = D^{\vhi^\sharp}(\vx \| \vx'). \]
\end{fact}

\begin{fact}[Corollary 1, \cite{AhnC21}]\label{fact:slc}
    Suppose $\pi$ is $\vhi$-strongly log-concave. Then, for all densities $\mu$ on $\R^d$, \[
    \dkl(\mu\|\pi)\ge W^\vhi(\mu\|\pi).
    \]
\end{fact}

The key assumption we rely on to derive our entropic contraction is the following.

\begin{assumption}\label{assume:llt}
Let $\psi \defeq \vhi^\sharp$, and let $V \succeq \alpha \psi$. Then $\alpha V^\sharp \preceq \vhi$.
\end{assumption}

To motivate Assumption~\ref{assume:llt}, we first consider the case $\alpha  = 1$ and $V = \psi$. Then Assumption~\ref{assume:llt} simply posits that $\psi^\sharp = \vhi^{\sharp\sharp}$ is ``less convex'' than $\vhi$, in the sense that $\vhi - \vhi^{\sharp\sharp}$ is convex. In other words, taking the LLT of a convex function $\vhi$ twice gives a function which is smaller than $\vhi$ in the convex partial ordering. There are some analogies present in the literature between the Fenchel conjugation operator $^*$ (for which $\vhi^{**} = \vhi$) and the LLT operator, e.g. Section B.2, \cite{ChenCSW22} and Section 3.1, \cite{GopiLLST25}. Unfortunately, essentially any non-quadratic example witnesses $\vhi^{\sharp\sharp} \neq \vhi$.\footnote{An easy way to see this is to take any $\vhi$ that is not self-concordant. Then $\vhi^{\sharp\sharp} = \vhi$ is ruled out simply by Fact~\ref{fact:llt-smooth-convex}.}

Assumption~\ref{assume:llt} suggests that the LLT operator satisfies a relaxed version of the property $\vhi^{\sharp\sharp} = \vhi$. More generally, the relative strong convexity and scaling in Assumption~\ref{assume:llt} is as one would expect based on the Gaussian case (Example~\ref{ex:llt_gaussian} below). Note that applying Fact~\ref{fact:llt-smooth-convex} to quadratic norms shows that Assumption~\ref{assume:llt} is indeed correct whenever $\vhi$ itself is a quadratic.

\begin{example}\label{ex:llt_gaussian}
    Let $\vhi(\vy) = \frac{1}{2}(\vy-\mu)^\top\msig^{-1}(\vy-\mu)$, so that the density $\propto \exp(-\vhi)$ is $\calN(\mu,\msig)$. Then, the density $\propto \exp(-\psi)$ is $\calN(-\mu,\msig^{-1})$.
\end{example}
\begin{proof}
    This is a simple calculation. Indeed, \begin{align*}
        \psi(\vx) &= \log\Par{\int\exp\Par{\inprod{\vx}{\vy}-\frac{1}{2}(\vy-\mu)^\top \msig^{-1}(\vy-\mu)}\dd\vy}\\
        &= \log\Par{\int\exp\Par{-\frac{1}{2}(\vy-\mu-\msig\vx)^\top \msig^{-1}(\vy-\mu-\msig\vx)+\inprod{\vx}{\mu}+\frac{1}{2}\vx^\top\msig\vx}\dd\vy}\\
        &=\inprod{\vx}{\mu}+\frac{1}{2}\vx^\top\msig\vx+C\\
        &=\frac{1}{2}(\vx+\mu)^\top\msig(\vx+\mu)+C',
    \end{align*}
    for constants $C,C'$ independent of $\vx$.
\end{proof}

We prove the convergence of Algorithm~\ref{alg:alternate} in KL divergence as follows.

\begin{theorem}\label{thm:under-conjecture}
Let $\vhi: \R^d \to \R$ be convex, $\psi \defeq \vhi^\sharp$, and $\pi=\tpi^X \propto \exp(-V)$ such that $V \succeq \alpha\psi$.  If Assumption~\ref{assume:llt} holds, then for any $\mu_0\ll \pi$, iterate $\vx_k$ of the proximal sampler (Algorithm \ref{alg:alternate}) has density $\mu_k$ satisfying
    \[
        \dkl(\mu_k \|\pi) \le \frac{1}{(1 + \alpha/\tau)^{2k-1}}\dkl(\mu_0\|\pi).
    \]
\end{theorem}
\begin{proof}
    This is a straightforward consequence of Lemmas~\ref{lem:w_psi_contraction} and~\ref{lem:kl_le_w_psi} below. In particular, we apply Lemma~\ref{lem:w_psi_contraction} for $k - 1$ iterations, and in the last iteration, we apply Lemma~\ref{lem:kl_le_w_psi}.
\end{proof}

\begin{lemma}\label{lem:w_psi_contraction}
    Instate the setting of Theorem \ref{thm:under-conjecture}.  Consider running two copies of Algorithm~\ref{alg:alternate} initialized with $\vx_0$ and $\vx'_0$, and with iterates $\{\vy_k, \vx_{k + 1}\}_{0 \le k < K}$ and $\{\vy'_k, \vx'_{k + 1}\}_{0 \le k < K}$, respectively.
Letting $\mu_{k}, \mu'_k$ denote the total laws of $\vx_k, \vx'_k$ on an iteration $k$, for any $0 \le k < K$ in Algorithm~\ref{alg:alternate},
\[W^{\tau\psi}(\mu'_{k + 1} \| \mu_{k + 1})  \le \frac{1}{(1+\alpha/\tau)^2}W^{\tau\psi}(\mu'_k \| \mu_k). \]
\end{lemma}
\begin{proof}
Throughout the proof, we let $\mu_{k + 0.5}, \mu'_{k + 0.5}$ denote the total laws of $\vy_k, \vy'_k$ on an iteration $t$. 
Recall the notation of \eqref{eq:joint-density}.  Let $\gamma_k$ be the coupling of $(\mu_k, \mu_k')$ realizing $W^{\tau\psi}(\mu_k' \| \mu_k)$. Conditioning on the values of $(\vx_k, \vx_k') = (\vx, \vx') \sim \gamma_k$, we have by Fact~\ref{fact:slc} and Fact~\ref{fact:kld_breg} respectively that
\begin{equation}\label{eq:ystep}
W^{\vhi^{*\tau}}\Par{\tpi_{(\tau)}^{Y\mid X=\vx} \,\big\|\, \tpi_{(\tau)}^{Y\mid X=\vx'}} \le \dkl\Par{\tpi_{(\tau)}^{Y\mid X=\vx} \,\big\|\, \tpi_{(\tau)}^{Y\mid X=\vx'}} = D^{\tau\psi}(\vx' \| \vx).
\end{equation}
Now let $\gamma_{k + 0.5}$ be the coupling of $(\mu_{k + 0.5}, \mu'_{k + 0.5})$ which first draws $(\vx, \vx') \sim \gamma_k$, and then conditioned on these values, draws and outputs $(\vy, \vy')$ from the coupling realizing $W^{\vhi^{*\tau}}(\tpi_{(\tau)}^{Y \mid X = \vx} \| \tpi_{(\tau)}^{Y|X=\vx'})$. 

Conditioning on the values of $(\vy, \vy') \sim \gamma_{k + 0.5}$, we have respectively by $V+\tau\psi\succeq (1+\frac\alpha\tau)\tau\psi$, Fact~\ref{fact:slc}, and Fact~\ref{fact:kld_breg} that
\begin{equation}\label{eq:firstpart_x}
\begin{aligned}
\Par{1 + \frac \alpha \tau}W^{\tau\psi}\Par{\tpi_{(\tau)}^{X\mid Y=\vy'} \,\big\|\, \tpi_{(\tau)}^{X\mid Y=\vy}} &\le W^{V + \tau\psi}\Par{\tpi_{(\tau)}^{X\mid Y=\vy'} \,\big\|\, \tpi_{(\tau)}^{X\mid Y=\vy}} \\
    &\le \dkl\Par{\tpi_{(\tau)}^{X\mid Y=\vy'} \,\big\|\, \tpi_{(\tau)}^{X\mid Y=\vy}} = D^{(V + \tau\psi)^\sharp}(\vy \| \vy').
\end{aligned}
\end{equation}
Now by applying Assumption~\ref{assume:llt} and Lemma~\ref{lem:lltconv}, we finally have
\begin{equation}\label{eq:secondpart_x}
\Par{1 + \frac \alpha \tau}D^{(V + \tau\psi)^\sharp}(\vy \| \vy') \le D^{\vhi^{*\tau}}(\vy \| \vy').
\end{equation}
Let $\gamma_{k + 1}$ be the coupling of $(\mu_{k + 1}, \mu'_{k + 1})$ which first draws $(\vy, \vy') \sim \gamma_{k + 0.5}$, and then conditioned on these values, draws and outputs $(\vx, \vx')$ from the coupling realizing $W^{\psi^{*\tau}}(\tpi_{(\tau)}^{X\mid Y=\vy'} \,\big\|\, \tpi_{(\tau)}^{X\mid Y=\vy})$. 

By combining \eqref{eq:ystep}, \eqref{eq:firstpart_x}, \eqref{eq:secondpart_x}, and using the fact that $W^{\tau\psi}(\mu'_{k + 1} \| \mu_{k + 1})$ is an infimum over couplings (one of which is $\gamma_{k + 1}$), we have shown $(1+\frac \alpha \tau)^2W^{\tau\psi}(\mu'_{k + 1} \| \mu_{k + 1}) \le W^{\tau\psi}(\mu_k' \| \mu_k)$, as claimed. 
\end{proof}

\begin{lemma}\label{lem:kl_le_w_psi}
    Consider the same setup as in Lemma~\ref{lem:w_psi_contraction}. It holds that \[
    \dkl(\mu_{k+1}'\|\mu_{k+1})\le \Par{1 + \frac \alpha \tau}W^{\tau\psi}(\mu_k'\|\mu_k).
    \]
\end{lemma}
\begin{proof}
Following Lemma~\ref{lem:w_psi_contraction} up to \eqref{eq:firstpart_x}, we have some coupling $\gamma_{k + 0.5}$ of $(\vy_k, \vy'_k) = (\vy, \vy')$ such that
\[\Par{1 + \frac \alpha \tau}\E_{(\vy, \vy') \sim \gamma_{k + \half}}\Brack{\dkl\Par{\tpi_{(\tau)}^{X \mid Y = \vy'} \,\big\|\, \tpi_{(\tau)}^{X \mid Y = \vy}}} \le W^{\tau\psi}\Par{\mu'_k \| \mu_k}.\]

    We now apply Fact~\ref{fact:lem3_lst21}, analogously to its use in Theorem 1, \cite{LeeST21}, with $\vz \gets (\vy, \vy')$.
\end{proof}

\begin{fact}[Lemma 3, \cite{LeeST21}]\label{fact:lem3_lst21}
    Let $P_\vz$ and $Q_\vz$ be distributions supported on $\calX$ indexed by $\vz$, a random variable distributed as $\pi_\vz$. Let $\Tilde{P}$ be the joint distribution of $(\vx,\vz)$ for $\vx\sim P_\vz$ and $\vz\sim\pi_\vz$ and $\Tilde{Q}$ be the joint distribution of $(\vx,\vz)$ as $\vx\sim Q_\vz$. Then, \[
    \dkl(P\|Q)\le \E_{\vz\sim\pi_\vz}[\dkl(P_\vz\|Q_\vz)].
    \]
\end{fact}

\section{Erratum Since COLT 2026 Version}\label{sec:erratum}
An earlier version of this manuscript was presented at COLT 2026, which claimed a Poincar\'e inequality for the dual $\tpi^Y_{(\tau)}$ based on the following incorrect claim.

\begin{claim}
    Suppose $\pi=\pi_1*\pi_2$, where $\pi_i$ satisfies an $\alpha_i$-$\varphi$-PI, $i=1, 2$.  Then $\pi$ satisfies an $\varphi$-PI with constant 
    \[\frac{1}{\alpha_1^{-1}+\alpha_2^{-1}}.\]
\end{claim}

This intermediate claim has been removed, and does not affect any of our main results. We provide an explicit counterexample, a modification of Hank Green's. We work in dimension 1, so we use scalars everywhere. Let $q(z) := \frac14 z^4 + \frac12 z^2$ and, for $a>0$, define
\[
  \vhi_a(y) \defeq q(y-a) + \log Z,
  \qquad
  Z \defeq \int_{\R} \exp(-q(z))\dd z.
\]
Set $\pi_1=\pi_2=\exp(-\vhi_a)$. By Brascamp-Lieb, both measures satisfy a $1$-$\phi_a$-PI.

Let $Z_1,Z_2\sim \exp(-q)$ be independent, and write $Y_i=a+Z_i\sim\pi_i$ for $i\in \{1,2\}$. Then $Y=Y_1+Y_2\sim\pi_1*\pi_2$. Applying the claimed $\frac12$-$\vhi_a$-PI to the linear function $f(y)=y$ would give
\[
\begin{aligned}
  2\Var(Z_1)
  &= \Var(Y)\\
  &\le 2\E\Brack{\frac{1}{\vhi_a''(Y)}}\\
  &=2\mathbb E\Brack{\frac{1}{1+3(a+Z_1+Z_2)^2}}.
\end{aligned}
\]
However, the right-hand side converges to zero as $a\to\infty$ by dominated convergence, while the left-hand side is a strictly positive constant. This is a contradiction.

\end{document}